\documentclass[english,leqno,11pt,a4paper]{amsart}

\usepackage{amsmath,amstext}
\usepackage{amsthm, amssymb}
\usepackage{endnotes}

\usepackage{pifont}
\usepackage{hyperref}
\usepackage{bookmark}

\usepackage{srcltx}
\usepackage[english]{babel}
\usepackage[all]{xy}

\usepackage[utf8]{inputenc}
\usepackage[T1]{fontenc}

\newtheorem{thmintro}{Theorem}

\newtheorem{corintro}[thmintro]{Corollary}
\newtheorem{lemintro}[thmintro]{Lemma}

\newtheorem{questionintro}[thmintro]{Question}

\newtheorem{thm}{Theorem}[section]
\newtheorem{cor}[thm]{Corollary}
\newtheorem{lemma}[thm]{Lemma}
\newtheorem{prop}[thm]{Proposition}

\theoremstyle{plain}
\newtheorem{defn}[thm]{Definition}
\newtheorem{Definition}[thm]{Definition}

\theoremstyle{definition}
\newtheorem{rmk}[thm]{Remark}

\numberwithin{equation}{section}

\def\beq{\begin{equation}}
\def\eeq{\end{equation}}

\def\Z{{\mathbb Z}}

\def\C{{\mathbb C}}

\def\l{\left}
\def\r{\right}

\def\ds{\displaystyle}

\newcommand{\ZZ}{\mathbb{Z}}
\newcommand{\FF}{\mathbb{F}}
\newcommand{\CC}{\mathbb{C}}
\newcommand{\QQ}{\mathbb{Q}}

\newcommand{\GG}{\mathbb{G}}

\newcommand{\NN}{\mathbb{N}}

\def\CC{{\mathbb{C}}}
\def\ZZ{{\mathbb{Z}}}

\def\cM{{\mathcal M}}

\def\cP{{\mathcal P}}

\def\sg{\sigma}

\def\la{\lambda}
\def\La{\Lambda}

\def\deg{\mathop{\rm deg}}

\def\GL{\mathop{\rm GL}}

\def\Cas{\mathop{\rm Cas}}

\def\$$endproof{\eqno{\qedhere}$$\end{proof}}

\makeindex

\title[The Carlitz module]{The Carlitz module and an Ax-Lindemann-Weierstrass Theorem related to Euler's gamma function}

\author{Lucia Di Vizio}
\address{Lucia Di Vizio,
Laboratoire de Math\'ematiques UMR 8100, CNRS,
Universit\'e de Versailles-St Quentin,
45 avenue des \'Etats-Unis
78035 Versailles cedex, France.\\
{\tt divizio@math.cnrs.fr}}
\author{Federico Pellarin}
\address{Federico PELLARIN, Dipartimento di Matematica Guido Castelnuovo, Universit\`a Sapienza, Piazzale Aldo Moro 5, 00185 Rome, Italy\\
{\tt federico.pellarin@uniroma1.it}}

\date{\today}

\begin{document}

\bibliographystyle{amsalpha}

\begin{abstract}
We prove a differential transcendence result of type ``Ax-Lindemann-Weierstrass'' for Euler's
gamma function. Given meromorphic functions $\zeta_1,\dots,\zeta_n$ of a complex variable $\nu$
that are pairwise distinct modulo $\Z$ and algebraic over the field $k$ of meromorphic $1$-periodic functions,
the functions
$ \Gamma(\nu-\zeta_1(\nu)),\dots,\Gamma(\nu-\zeta_n(\nu))$ are differentially independent over the field $k(\nu)$.
\par
We determine the structure of certain difference field extensions related to the torsion
of an avatar of the Carlitz module over meromorphic functions. These extensions are abelian and purely transcendental, the latter property being crucial in our main result, and obtained applying a criterion of differential algebraicity of Hardouin and Singer.
\end{abstract}

\maketitle

\setcounter{tocdepth}{1}
\tableofcontents


\section{Introduction}

Let $k$ be the field of $1$-periodic meromorphic functions over $\C$.
The identity function
$\nu\mapsto\nu$ over $\CC$
is transcendental over $k$, and we consider the field of rational
functions $K:=k(\nu)$, which is a subfield of the field $\cM$ of meromorphic functions over $\C$.
We use an analogue of Carlitz module, introduced in \cite{pellarinGeneralizedCarlitzModule2013},
for the $k$-linear finite difference operator
\begin{equation}\label{definition-tau-difference}
\tau:f(\nu)\mapsto f(\nu+1),
\end{equation}
to study certain abelian extensions of $K$
generated by the torsion. Such extensions have a deep connection with Euler's gamma function, which is
the meromorphic function over $\CC$ defined, for $\nu\in\CC\setminus\{-1,-2,\ldots\}$ by
    $$
    \Gamma(\nu):=\frac{1}{\nu}\prod_{m\geq1}\left(\Big(1+\frac{\nu}{m}\Big)^{-1}\Big(1+\frac{1}{m}\Big)^{\nu}\right).
    $$
Using this theory and a criterion of differential independence by Hardouin and Singer \cite{hardouinDifferentialGaloisTheory2008},
we prove the following result of type ``Ax-Lindemann-Weierstrass'' for Euler's
gamma function. See Theorem~\ref{thm:Holder-ALW} below:

\begin{thmintro}\label{thm:Holder-ALW-Intro}
Let $\zeta_1,\dots,\zeta_n$ be algebraic over $k$ and pairwise distinct modulo $\Z$.
Then the functions $ \Gamma(\nu-\zeta_1(\nu)),\dots,\Gamma(\nu-\zeta_n(\nu))$ are differentially independent over $K=k(\nu)$.
\end{thmintro}

In the statement above: (1) for any function $\xi$ algebraic over $K$ and any meromorphic function $f$ over $\CC$, the composition
        $\nu\mapsto f(\xi(\nu))$ can be viewed as a meromorphic function over a suitable non-empty open subset of $\CC$,
(2) meromorphic functions
        $f_1,\ldots,f_n$ defined over a non-empty open subset of $\CC$ are {\em differentially independent over
        the base field $K$}
        if there is no non-trivial algebraic differential dependence relation over $K$ between
        them, with respect to the derivation $D:=\frac{d}{d\nu}$.

We justify Ax-Lindemann-Weierstrass in the following way. Tracing analogies between differentially algebraic functions
and algebraic numbers, the functions $\zeta_i(\nu)$ and $\nu-\zeta_i(\nu)$ being differentially algebraic, they play the role of algebraic numbers for the classical Lindemann-Weierstrass Theorem (for evaluations of the exponential, see \cite{waldschmidt1983} for an account of the theory). At once, in Theorem \ref{thm:Holder-ALW-Intro}, we compose $\Gamma$ with certain meromorphic functions, which is reminiscent of  Ax's seminal paper \cite{ax1971schanuel}, addressing one of the first  functional problems in this range, for the exponential function.

\smallskip

We present some consequences of our result.
With $n$ an integer $>1$, we consider the constant functions $$\zeta_1=0\quad\zeta_2=\frac{1}{n}\quad\cdots\quad\zeta_n=\frac{n-1}{n},$$ that are pairwise distinct modulo $\ZZ$. Our theorem implies

\begin{corintro}
The functions
$\Gamma\big(\nu+\frac{k}{n}\big)$, for $k=0,\ldots,n-1$, are differentially independent over $K$.
\end{corintro}

For $n=1$ this implies Hölder's theorem \cite{holder1886ueber}, stating that the gamma function
is differentially transcendental over the field  $\CC(\nu)$ of rational functions over $\CC$, that is, $\Gamma(\nu)$ and all its derivatives in $\nu$ are algebraically independent over $\CC(\nu)$.

Notice that taking termwise logarithmic derivative of the multiplication formula
\begin{equation}\label{multiplication-formula}
\prod_{k=0}^{n-1}\Gamma\Big(\nu+\frac{k}{n}\Big)=(2\pi)^{\frac{n-1}{2}}n^{\frac{1}{2}-n\nu}\Gamma(n\nu)
\end{equation}
provides a differential relation over $K$ for $\Gamma\big(\nu+\frac{k}{n}\big)$, for $k=0,\ldots,n-1$ and $\Gamma(n\nu)$
but it does not prove algebraic dependence of these functions because of the factor $n^{\frac{1}{2}-n\nu}$, transcendental over $K$.

\smallskip
An original feature of our result is the functional transcendence over the field $K$ and we can allow $\zeta_1,\ldots,\zeta_n$ non-constant (so, unlike other Ax-Lindemann-Weierstrass statements, our base field is $k$, not $\CC$). To illustrate this choose $\zeta\in k$ and suppose that for all $0<r<1$, $\zeta$ is not $r$-periodic. In particular $\zeta$ is not a constant function.
Now consider an integer $n>0$ and, for $i=1,\ldots,n$, the meromorphic functions
$$\zeta_i(\nu)=\zeta\Big(\frac{\nu+i-1}{n}\Big),$$ that are algebraic over $k$. More precisely, for all $i$, $k[\zeta_i]$ equals the subfield
of meromorphic $n$-periodic functions over $\CC$ which is Galois cyclic of order $n$ over $k$, and the set of these functions is Galois invariant (see Lemma \ref{lemma-5}). The field $k[\zeta_i]$ is the splitting field of the polynomial
$$P_n(X)=\prod_{j=1}^{n}\big(X-\zeta_j\big)\in k[X].$$ One sees easily that $\zeta_1,\ldots,\zeta_n$ are distinct modulo $\ZZ$.
Indeed otherwise, we would have $\zeta'$ periodic with period $r$ for some $0<r=\frac{k}{n}<1$, $k\in\ZZ$ which easily contradicts
the fact that the minimal positive period of $\zeta$ is one (this last step also uses that $\zeta$ is meromorphic).

For example, in the case $\zeta(\nu)=\cos(2\pi\nu)$, we have
$$2^{n-1}P_n(X)=T_n-\zeta$$
where $T_n\in\ZZ[X]$ is the $n$-th Chebyshev polynomial of the first kind, characterized by
$T_n(\cos(2\pi\nu/n))=\cos(2\pi\nu)$.

Our Theorem \ref{thm:Holder-ALW-Intro}
implies

\begin{corintro}
Over a suitable non-empty open subset of $\CC$ where all the functions are well defined and meromorphic,
the $n$ functions $$\Gamma\Big(\nu-\zeta\big(\frac{\nu}{n}\big)\Big),\ldots,\Gamma\Big(\nu-\zeta\big(\frac{\nu+n-1}{n}\big)\Big)$$
are differentially independent over $K$.
\end{corintro}

One could ask to go a step further, and consider the problem of classifying differential relations among
functions of the type $\Gamma(\alpha\nu+\zeta)$ with $\alpha,\zeta$ algebraic over $k$ (Theorem \ref{thm:Holder-ALW-Intro}
only deals with the case $\alpha=1$). In this slightly generalized setting it is well known that non-trivial algebraic relations occur, that do not follow from
the standard functional equation $\Gamma(\nu+1)=\nu\Gamma(\nu)$.
The reflection formula
\begin{equation}\label{reflection-formula}
\Gamma\big(\nu\big)\Gamma\big(1-\nu\big)=\frac{\pi}{\sin(\pi\nu)}
\end{equation}
asserts that $\Gamma(\nu)\Gamma(-\nu)$ is algebraic over $k$, being a root of the irreducible polynomial
$$X^2-\frac{\pi^2}{\nu^2\sin(\pi\nu)^2}\in k[X].$$

As already noticed, for any $n\geq 2$, the multiplication formula
\eqref{multiplication-formula}
yields a non-trivial differential relation for the $n+1$ functions $\Gamma\big(\nu+\frac{k}{n}\big)$ for $k=0,\ldots,n-1$ and $\Gamma(n\nu)$.
The following question naturally arises:

\begin{questionintro}\label{question:intro}
Given functions $\xi_1,\ldots,\xi_n$ differentially algebraic over $K$,
determine elementary conditions
so that the functions
    $$
    \Gamma\big(\xi_1\big),\ldots,\Gamma\big(\xi_n\big),
    $$
seen as analytic functions over some non-empty open subset of $\CC$,
are differentially transcendental over $K$.
\end{questionintro}
Intuitively, Question~\ref{question:intro} connects to the possibility of a functional and differential variant of R\"ohrlich 1970 Conjecture (containing a conjectural classification of all the multiplicative relations as a consequence of the so-called Deligne-Koblitz-Ogus criterion) or of Lang-R\"ohrlich Conjecture
\cite[Appendix to \S 2, p. 66]{langCyclotomic} (see also
\cite{langRelationsDistributionsExemples1977} and \cite{GunMurtyRath2011}) where the role of the field $\QQ$ is here played by the field $K$.
One can guess if ``elementary condition'' may just signify avoiding relations arising from combinations of the relations \eqref{reflection-formula} and \eqref{multiplication-formula} in a way predicted by a possible functional analogue of Deligne-Koblitz-Ogus criterion or Deligne's reciprocity law as in \cite[Theorem 7.15]{DeligneMilneOgusShih1982}.
Without such a functional criterion it is legitimate to refrain from formulating a conjecture explicitly, but the Carlitz module involved in our work is the first sign that such structures may be realized.
\par
We also mention that refined analytic methods also allow to study algebraic relations among special values of the gamma function, with the base field $\CC$, see the work of Eterovi\'c and Padgett \cite{eterovicEquationsInvolvingGamma2025}, especially their Theorem 1.1. These results apply to special values of the functions $\Gamma(\alpha\nu+\zeta)$ with $\alpha\in\CC^\times,\zeta\in\CC$ (constant functions) but focus on questions which are different from those that are considered in this paper.

\subsection*{Background}
Problems in the wake of the Ax-Schanuel theorem have witnessed major advances and have important consequences, notably in the study of unlikely intersections and the Zilber–Pink conjecture. In this context, a strategy originally proposed by Zannier and later implemented by Pila and Zannier to provide an alternative proof of the Manin-Mumford conjecture in \cite{PilaZannier2008} led to the proof of several fundamental results in the area. See, for example, \cite{Pila22,Scanlon2022,Zannier12}, which provide useful background on these developments together with a broader bibliography than that included in the present paper. The strategy by Pila and Zannier combines point-counting results in $o$-minimal structures with functional transcendence theorems and other Diophantine ingredients. A major recent breakthrough in the area through this method is the proof of the general André–Oort conjecture \cite{PST}, using large-Galois orbit results obtained in \cite{BSY}, and functional transcendence results for Shimura varieties (see \cite{BT, GK, KUY, MPT, PT}) as one of the key ingredients.
The base field is usually $\mathbb{C}$, although other fields
are also considered, such as $\mathbb{C}_p$ (the completion of an algebraic
closure of $\mathbb{Q}_p$ for a prime number $p$).

We add that, in order to address these questions, one often either exploits the fact that $\Phi$ satisfies partial differential equations and applies differential Galois theory (see, for example, \cite{BlazquezSanz2021AxSchanuel,casale2020ax}), or adopts the framework of bi-algebraicity together with the strategy of “counting points in an $o$-minimal structure” via the Pila–Wilkie theorem \cite{PW}. The latter applies when the considered special function $\Phi$ exhibits suitable periodicity properties, as in the case of exponential functions; see \cite{MPT}. This is not easily applicable in our case, as $\Phi=\Gamma$ is not periodic, and more fundamentally, a certain restriction of $\Gamma$ is not definable in any $o$-minimal structure, see Padgett and Speissegger
\cite[Prop.~27]{PadgettSpeissegger2025}.

\subsection*{Approach}

To prove our Theorem \ref{thm:Holder-ALW-Intro} we introduce a completely different third approach by studying a variant of Carlitz module over meromorphic functions introduced in \cite{pellarinGeneralizedCarlitzModule2013}.
Let us consider the polynomial ring $A:=k[t]$, where $t$ is an indeterminate. There exists a unique
$k$-algebra homomorphism
$$A\xrightarrow{C}k[\nu][\tau]$$
defined by $t\mapsto \nu-\tau$ where the algebra structure of $k[\nu][\tau]$ is determined by $\tau x=x\tau $, for all $x\in k$, and $\tau \nu=(\nu+1)\tau$.

The morphism $C$ defines a functor $B\mapsto C(B)$ from the category of $K[\tau]$-modules to
the category of $A$-modules. For example, $C(\mathcal{M})$ is the $k$-vector space of meromorphic functions with the structure of Carlitz $A$-module described above.
If $a\in A$ we denote by $C_a\in k[\nu][\tau]$ the image of $a$ by $C$. In this way we have the collection of
pairwise commuting difference operators in $k[\nu][\tau]$.
Remarkably they all have a commutative difference Galois group, as a consequence of our results. For example we have:

\begin{corintro}\label{thmintro:Artinsymbol}
Assume $P\in A$ irreducible. The $\tau$-difference Galois group of the operator $C_P$ is canonically isomorphic to  the multiplicative group $(A/(P))^\times$.
\end{corintro}

A more precise and general result is proved in this paper, see our Theorem \ref{thm:Artinsymbol}.
The {\em direct problem in Galois theory of functional equations}, {that is the calculation
of the Galois group given the equation}, is considered to be difficult:
in the case of differential equations, Kolchin mentioned it as one of the most difficult problems in differential algebra
\cite{kolchinProblemsDifferentialAlgebra1968}. Significant progress has been made on this side, but it is quite rare to be able to
produce such a simple description.

The gamma function comes into the picture in the following way. When we consider an irreducible polynomial $P\in A=k[t]$ the $k$-vector space
$$\La_P=\{f\text{ meromorphic}:C_P(f)=0\}$$ has dimension $\deg_t(P)$ by a theorem of Praagman
\cite[Theorem 1]{praagman_fundamental_1986}. $\La_P$ is also an $A$ module via $C$. Proposition
\ref{prop:PowerOfIrriducibleTorsion} furnishes a procedure to compute a basis of $\La_{P^r}$,
for all $r\geq 1$, starting from a basis of $\La_P$. This result crucially uses the fact that the operator $D=\frac{d}{d\nu}$
commutes with $\tau$ and is a novel feature of Drinfeld modules over meromorphic functions that we use in our paper and does not exist in
classical function field arithmetic (so the theory we introduce in this paper can only partially be viewed as an analogue of the theory of Drinfeld modules in function field arithmetic).

Let $\zeta_1,\dots,\zeta_n$ be the roots of $P$ in an algebraic closure of $k$. We define the
following functions (see Lemmas \ref{lemma-y-solutions} and \ref{linear-independence}):
    \[
    y_r:=\sum_{i=1}^{n}\zeta_i(\nu)^r\Gamma\big(\nu-\zeta_i(\nu)\big),~\hbox{with $r=0,\dots,n-1$,}
    \]
meromorphic in a certain domain of $\C$.
A technical statement of our work is the partial determination of the domain of meromorphy:

\begin{lemintro}\label{lem:inversion}
There exists a discrete subset $S$ of $\CC$, invariant by translations by elements of $\ZZ$, such that, with $\Theta=\CC\setminus S$, $k_\Theta$ the field of $1$-periodic functions meromorphic over $\Theta$ and $\mathcal{M}_\Theta$ the field of meromorphic
functions over $\Theta$, $y_1,\ldots,y_n$ can be identified with meromorphic functions over $\mathcal{M}_\Theta$,
and are a basis of the $k_\Theta$-vector space $\{f\in\mathcal{M}_\Theta:C_P(f)=0\}$.
\end{lemintro}

The set $S$ is a subset of the set of poles of the coefficients of $P$, if we take $P$ monic. We cannot rule out the hypothesis that it is empty for certain $P$, or all, it is something we do not control, which is a puzzling problem from the analytic viewpoint.

The set $S$ is empty in the case $P=t$. Indeed in this case,
one sees easily that $C_t(\Gamma)=0$ so $\Theta=\CC$ and
$\La_t=k\Gamma$, so $\Gamma$ is a $t$-torsion point for Carlitz module over meromorphic functions.
Notice, for general $P$, that Praagman's Theorem \cite[Theorem~1]{praagman_fundamental_1986} computes the dimension over $k$ of $\La_P$ with a non-constructive proof. A more constructive process is given by Nörlund
\cite[end of \S5 and end of \S6]{norlundLexistenceSolutionsDune1914}
(see also \cite[Chapter 10, page 272]{noerlundVorlesungenUeberDifferenzenrechnung1924}): he constructs
a basis of solutions of linear difference equations that are meromorphic when the coefficients are rational but may have essential singularities with trivial monodromy,
if the coefficients are just meromorphic functions. Nor does Nörlund control the locus of essential singularities.

Lemma \ref{lem:inversion} plays the role of an analogue of Fourier inversion formula for Gauss sums, and the functions
$\Gamma(\nu-\zeta_i(\nu))$ of analogues of (suitably normalized) Gauss sums. This is a crucial remark in our work and connects to the papers \cite{carlitz1935polynomials,hayes1974explicit,Thakur1988,anglesUniversalGaussThakurSums2015} in the following way.
We recall that the usual Carlitz module is a special example of a Drinfeld module earlier introduced by Carlitz in \cite{carlitz1935polynomials}. The books \cite{goss2012basic,papikianDrinfeld2023} provide good accounts of the theory.
From Hayes' work \cite{hayes1974explicit} it emerges that the field extensions generated by torsion elements of Carlitz module are function field analogues of cyclotomic fields. In his paper, Hayes uses it to prove a function field variant of Kronecker-Weber's theorem.
A reformulation of Hayes' theorem was reached in \cite{anglesUniversalGaussThakurSums2015} through the special values of
{\em Anderson-Thakur omega function}. The key point is that the omega function interpolates at the infinity valuation
suitably normalized Gauss-Thakur sums, introduced in \cite{Thakur1988}. In \cite{pellarinGeneralizedCarlitzModule2013} it was noticed that this makes an analogy with the two variable function $\Gamma(\nu-t)$. Therefore in the present paper, we successfully
applied these remarks to concrete results that can be considered as a first instance of difference class field theory, where a function related to Euler's gamma function plays a crucial role, just as the Anderson-Thakur's omega function in function field arithmetic. As an example of an immediate corollary of our Theorem \ref{thm:TrdegPV}:

\begin{corintro}
The algebraic closure $k_\infty$ of $k$ in the field of meromorphic functions $\mathcal{M}$ and $$K(\La_\infty):=K\Big(\bigcup_{a\in k[t]}\La_a\Big)\subset\mathcal{M}$$ are linearly disjoint in $\mathcal{M}$.
\end{corintro}

A difference variant of class field theory for the field $K$ is at the time being not fully available but in several aspects our work highlights structures that
point in that direction. This is natural and based on the idea, already suggested by Drinfeld, Krichever, Mumford \cite{kricheverAlgebraicCurvesCommuting1976,kricheverMethodsAlgebraicGeometry1977,mumfordAlgebrogeometricConstructionCommuting1978}, that the torsion of rank one Drinfeld modules (or Krichever modules) can be used to classify abelian linear difference (or differential) equations.

We end the introduction by observing that a complete description of the maximal abelian difference extension of $\CC(\nu)$ was carried out by the authors of
\cite{putKricheverModulesDifference2004}
but the fact that they use the field $\CC$ in place of our $k$ makes it unlikely that their methods can reach the general questions we are evoking in our
paper and certainly forbids deducing our main theorems from their results in an easy way.

\subsection*{Links with function field arithmetic}
We point out that the introduction of ``Carlitz module over meromorphic functions'' in Ax-Lindemann-Weierstrass-type results for $\Gamma$ is not a mere translation of Carlitz-Hayes theory and there are major difficulties and subtleties we had to overcome to reach our results. In particular: (1) differential structures appear, which are compatible with the difference structures of Carlitz module, and do not exists in Carlitz-Hayes theory (\S \ref{section-carlitz}). (2) It is non-trivial to connect $\Gamma$ with the $P$-torsion $\La_P$ with $P\in A$  of degree $>0$ (\S \ref{ALW-gamma}), for instance, in the classical Carlitz-Hayes theory, torsion modules are easily described through cyclotomic values of Carlitz's exponential (see Remark \ref{rmk:descriptionofthetorsion}) but there is no such a structure here; we rather need a variant of Fourier inversion theorem for Gauss sums, see Lemma \ref{lem:inversion} or \ref{lemma-y-solutions}. (3) An important tool of our work is Praagman's \cite[Theorem 2]{praagman_fundamental_1986}, rooted on the fact every holomorphic line bundle over $\CC^\times$ is trivial, which uses Runge's approximation theorem, not constructive. This yields delicate analytic problems related to the
functions $y_r$, see Lemma \ref{lemma-tau-S}. (4) Although the proof of our main result does not use the structure of difference Galois groups, checking the setup of the criterion of Hardouin and Singer
Proposition \ref{prop:HS} is subtle (unlike the simpler theorem of Hölder) and passes through a study of an auxiliary ring $F$ borrowing technical tools from Ovchinnikov and Wibmer's \cite{ovchinnikovGaloisTheoryLinear2014}, see for example Lemma \ref{lemma-subtle} and the proof of Proposition~\ref{prop:pinfty}.

Our results look optimal with the methods employed, notably (the avatar over meromorphic functions of) Carlitz's module. In particular, to proceed further and approach an answer to Question \ref{question:intro}, general ``motive'' structures are needed. In connection with this we recall
that in \cite{andersonDeterminationAlgebraicRelations2004}, Anderson, Brownawell and Papanikolas succeeded in demonstrating
a function field analogue of the ``rational case'' of Lang-R\"ohrlich conjecture for the {\em geometric gamma function} of Thakur, see their Theorem 6.2.1. In their proof they introduced and used several tools, including novel linear independence criteria, and a class of {\em dual $t$-motives} of ``CM type'' allowing to algebraically realize a converse of their diamond bracket criterion. Although we cannot consider
the geometric gamma function above as a complete analogue of Euler's gamma function, we can hope that some notion of
``dual $t$-motives over meromorphic functions'',
more general than our Carlitz module, can be used to progress in the direction of a complete answer to
Question \ref{question:intro}.

\subsection*{\sc Acknowledgements.}
We both thank the Vietnam Institut for Advanced Studies in Mathematics for the invitation in 2018, which was a turning point in
our collaboration, and the Fondation Mathématique Jacques Hadamard that have supported a visit of the second author to the
Université de Versailles Saint-Quentin-en-Yvelines during Spring 2025. We would like to thank Alin Bostan for his interest in our
work, and Michael Singer for appropriate questions and insightful discussions, as well as for a careful reading that allowed us to
improve the quality of our presentation.
We are thankful to Laura Capuano and Marc-Hubert Nicole for useful comments and suggestions.

\section{Difference algebra background}
\label{sec:background}

In the rest of this paper all fields are of characteristic zero.
A {\em difference field} (or a {\em $\tau$-field}) is a pair $(K,\tau)$, where $K$ is a field and $\tau:K\to K$ is an endomorphism.
We say that $(K,\tau)$ is {\em inversive} if $\tau$ is an automorphism.
We denote by $$k:=K^\tau:=\{x\in K:\tau(x)=x\}$$ the subfield of $\tau$-invariant elements of $K$,
also called $\tau$-constants.
Similarly, we call $\tau$-ring a ring equipped with an endomorphism $\tau$ and $\tau$-ideal an ideal of a $\tau$-ring which stable by $\tau$.
A $\tau$-ring is said to be $\tau$-simple if its only $\tau$-ideals are $(0)$ and the whole ring.
\par
Let $x_1,\ldots,x_s$ be elements of $K$. Their {\em $\tau$-casoratian} is the matrix:
    \begin{equation}\label{eq:tauwronskian}
        \hbox{$\Cas_\tau$}(x_1,\ldots,x_s)=
        \left(\begin{array}{cccc}
        x_1 & x_2  & \cdots &  x_s \\
        \tau (x_1) & \tau (x_2) & \cdots & \tau (x_s)\\
        \vdots & \vdots & & \vdots \\
        \tau^{s-1}(x_1)& \tau^{s-1}(x_2) & \cdots & \tau^{s-1}(x_s)
        \end{array}\right).
    \end{equation}
It is an analog of the wronskian matrix and has similar properties:

\begin{lemma}[{\cite[\S 7]{casorati1880calcolo},\cite[Lemma 8.2.1]{levinDifferenceAlgebra2008}}]
\label{lemma:tauwronskian}
    The elements $x_1,\ldots,x_s\in K$ are $k$-linearly independent if and only if
    the matrix $\Cas_\tau(x_1,\ldots,x_s)$ has maximal rank.
\end{lemma}

For any difference field $(K,\tau)$ as above, we consider the non-commutative $K$-algebra $K[\tau]$
of finite linear combinations of powers of $\tau$ with coefficients in $K$ of the form
$f=\sum_{i\geq 0}f_i\tau^i$, with the obvious addition and the multiplication defined by the rule $\tau f=\tau(f)\tau$, for any $f\in K$.
The algebra $K[\tau]$ has a right division algorithm, hence any left ideal of $K[\tau]$ is principal. If $(K,\tau)$ is inversive,
then there also is a left division algorithm, hence in this case right ideals of $K[\tau]$ are principal as well.

\par
We consider a $\tau$-extension $(F,\tau)$ of $(K,\tau)$, i.e., a field extension $F/K$, with an extension of the endomorphism $\tau$ to $F$.
If $\mathcal{P}=\sum_{i=0}^np_i\tau^i\in K[\tau]$ and $c$ is an element of $F$, the {\em evaluation of $\mathcal{P}$ at $c$} defines a $k$-linear endomorphism of $F$
that we also denote by $\mathcal{P}$:
\begin{equation}\label{prototypeoff}
\mathcal{P}:c\mapsto \mathcal{P}(c):=\sum_{i=0}^np_i\tau^i(c)\in F.
\end{equation}
The kernel $\ker_k(\mathcal{P})$ is thus a sub-$k$-vector space of $F$, and, as a consequence of Lemma~\ref{lemma:tauwronskian}, we have $\dim_k\ker_k(\mathcal{P})\leq n$.
Let $y$ be an indeterminate. We call
\begin{equation}\label{eq:eq}
\mathcal{P}(y):=p_0y+p_1\tau(y)+\cdots+p_n\tau^n(y)=0
\end{equation}
the {\em linear $\tau$-difference equation} or {\em $\tau$-equation in the indeterminate $y$} associated to $\mathcal{P}$.
\par

\begin{Definition}\label{definition-difference-closed}{\em
The difference field $(F,\tau)$ is {\em linearly $\tau$-closed} for (or simply {\em linearly closed}) if for all $\mathcal{P}\in F[\tau]$
and all $f\in F$ the associated
non-trivial inhomogeneous linear $\tau$-difference equation $\mathcal{P}(y)=f$ has at least one non-zero solution in $F$.}
\end{Definition}

\begin{rmk}
We do not suppose that a linearly $\tau$-closed field is also algebraically closed.
\end{rmk}

The following two lemmas are well-known and elementary. We recall them for the reader's convenience:

\begin{lemma}
If $(F,\tau)$ is linearly closed then it is inversive.
\end{lemma}

\begin{proof}
We need to show that, for any non-zero $f\in F$,
the equation $\tau(y)=f$ has a solution in $F$.
Any solution of $\tau(y)=f$ is also a solution of the linear $\tau$-difference equation:
\begin{equation}\label{eq:inversive}
f\tau(y)-\tau(f)y=0
\end{equation}
and the vector space of solutions of the latter coincides with $kf=\{cf:c\in k\}$.
By assumption the equation \eqref{eq:inversive} has a non-zero solution $g\in F$. It follows that $\tau(g)$ is solution of $f\tau(y)-\tau(f)y=0$
and hence that there exists
$c\in k$ such that $\tau(g)=cf$. Since $\tau$ is an endomorphism and $g\neq 0$,
we have that $c\neq 0$ and hence that $\tau(c^{-1}g)=f$, which ends the proof.
\end{proof}

\begin{lemma}
The following assertions are equivalent:
\begin{enumerate}
  \item $(F,\tau)$ is linearly closed.
  \item Any difference equation of the form $\mathcal{P}(y)=0$ as in \eqref{eq:eq},
   with $p_0p_n\neq 0$, has a set of $n$ $k$-linearly independent solutions in $F$, that is, $\dim_k(\ker_k(\mathcal{P}))=n$.
\end{enumerate}
\end{lemma}

\begin{proof}
Let us prove that (1) implies (2).
The implication is true for $n=1$ by definition. Let us prove the statement by induction on $n>1$.
If $g\in F$ is a non-zero solution of $\mathcal{P}(y)=0$,
then $\mathcal{P}=\tilde{\mathcal{P}}(g\tau-\tau(g))$, for some $\tilde{\mathcal{P}}=\sum_{i=0}^{n-1}\tilde p_i\tau^i\in F[\tau]$, with $\tilde p_0\tilde p_{n-1}\neq 0$.
By induction, $\tilde{\mathcal{P}}(y)=0$ has $n-1$ solutions $f_1,\dots,f_{n-1}$ which are linearly independent over $k$. There exist $g_1,\dots,g_{n-1}\in F$ such that $g\tau(g_i)-\tau(g)g_i=f_i$. We need to prove that $g,g_1,\dots,g_{n-1}$ are linearly independent over $k$. By contradiction, let us suppose that there exist $\lambda,\lambda_1,\dots,\lambda_{n-1}\in k$ not all zero such that $\lambda g+\sum_{i=1}^{n-1}\lambda_ig_i=0$.
Applying the operator $g\tau-\tau(g)$,
we obtain a non trivial linear dependence relation of the $f_i$ and
we conclude that all the $\lambda_i$ must be zero. But then also $\lambda=0$, and we have found a contradiction.
\par
Let us now assume (2). The $k$-vector space of solutions of the homogeneous difference equation
$(f\tau-\tau(f))\circ\mathcal{P}(y)=0$ has dimension $n+1$ and therefore there exists $g\in F$
solution of $(f\tau-\tau(f))\circ\mathcal{P}(y)=0$ but not solution of $\mathcal{P}(y)=0$.
Then  $\mathcal{P}(g)$ is a non zero element of $F$ such that $f\tau(g)-\tau(f)g=0$.
There exists $c\in k$ such that $\mathcal{P}(g)=c f$, and hence $\mathcal{P}(c^{-1}g)=f$.
This proves the lemma.
\end{proof}

Assume that $(F,\tau)$ is linearly $\tau$-closed with $F$ a $\tau$-field extension of $K$, and that
$F^\tau=K^\tau=k$.
Let $\mathcal{P}=\sum_{i=0}^np_i\tau^i\in K[\tau]$, with $p_0p_n\neq0$. Then there exists a $k$-basis $y_1,\dots,y_n$
of $$\La_\mathcal{P}:=\ker_k(\mathcal{P})\subset F.$$
By Lemma~\ref{lemma:tauwronskian}
the matrix $\Cas_\tau(y_1,\dots,y_n)$ is invertible with coefficients in  the smallest $\tau$-field contained in $F$ and containing
$K(\La_\mathcal{P})$, which has the form
$K(\La_\mathcal{P},\tau(\La_\mathcal{P}),\dots,\tau^{n-1}(\La_\mathcal{P}))$.
The columns of $\Cas_\tau(y_1,\dots,y_n)$
are solutions of the $\tau$-linear difference system:
\begin{equation}\label{eq:sys}
\tau (y)=M_\mathcal{P}y,
\hbox{~where
$M_\mathcal{P}=\begin{pmatrix}
  0& 1& & \\
  \vdots & & \ddots& \\
  0& & & 1 \\
  -\frac{p_0}{p_n} & -\frac{p_1}{p_n} &\cdots & -\frac{p_{n-1}}{p_n}
\end{pmatrix}\in\GL_n(K)$,}
\end{equation}
where the indeterminate function $y$ is a column vector to which $\tau$ is applied componentwise.
Therefore $\Cas_\tau(y_1,\dots,y_n)$ is an invertible matrix solution of $\tau(Y)=M_PY$, where $Y$ is an indeterminate square matrix of order $n$.
Taking the determinant on both side, we obtain that $$\delta:=\det(\operatorname{Cas}_\tau(y_1,\dots,y_n))$$ satisfies $$\tau(\delta)=(-1)^n\frac{p_0}{p_n}\delta.$$
Notice that $\delta\in\cM$ and depends on $y_1,\dots,y_n$ only up to a multiplicative constant in $k$.
\medskip

To conclude this presentation of introductory material in difference algebra, we shortly discuss the subfield of $\tau$-periodic elements of $F$.
For any positive integer $n$, the $n$-dimensional $k$-vector space
    $$
    k_n:=\ker_k(1-\tau^n)\subset F
    $$
is a subfield of $F$. In particular, $k=k_1$.

\begin{lemma}\label{lemma-5}
For any $n>0$, the extension $k_n/k$ if finite of degree $n$, Galois, with cyclic Galois group which acts via iterations of $\tau$.
\end{lemma}

\begin{proof}
We have that $k_n\setminus \cup_{i<n}k_{i}\neq \emptyset$ because $F$ is linearly closed. It is separable.
If $f\in k_n$ then the polynomial
    $$
    P=\prod_{i=0}^{n-1}(X-\tau^i(f))
    $$
vanishes at $f$. Notice that $P\in k[X]$, since its coefficients are symmetric polynomials in $f,\dots,\tau^{n-1}(f)$.
If $f\in k_n\setminus\cup_{i<n}k_{i}$, $P$ is the minimal polynomial of $f$ and the extension $k(f)/k$ is
Galois of degree $n$, with cyclic Galois group generated by $\tau$, and the action is given as described above.
We conclude by noticing that $k_n$ being the set of solution of $\tau^n(y)-y=0$, it is a $k$-vector space of dimension $n$, hence $k[f]=k_n$.
\end{proof}

\begin{cor}\label{cor:kinfty}
The field $k_\infty:=\bigcup_{n>0}k_n\subset F$ is relatively algebraically closed in $F$.
\end{cor}

\begin{proof}
Let $f\in F$ be algebraic over $k$. Then it is separable with minimal polynomial $P\in k[X]$. The set $\{\tau^i(f):i\in\ZZ\}$
is a set of roots of $P$ so it is finite and there exists $i\neq j$ such that $\tau^i(f)=\tau^j(f)$. This means that there exists $n>0$ such that $f\in k_n$.
\end{proof}

\section{The Carlitz module over meromorphic functions}\label{section-carlitz}

We recall that $\mathcal{M}$ denotes the field of meromorphic functions over $\CC$.
We denote by $\nu$ both the identity map on $\CC$ and a variable on $\CC$ and we define
the automorphism $\tau$ on $\cM$ by
    $$
    \tau(f)(\nu)=f(\nu+1),\quad f\in\cM,\quad \nu\in\CC.
    $$
The field $k=\cM^\tau=\{f\in\mathcal{M}:\tau(f)=f\}$ being equal to
the field of meromorphic $1$-periodic functions, the field
$k_\infty\subset\cM$ is the subfield of the meromorphic functions over $\CC$ that are $n$-periodic for some positive integer $n$.
We observe that $\nu$ is transcendental over $k_\infty$, since the latter is relatively algebraically closed in $\cM$, by Corollary~\ref{cor:kinfty}.
Our base field is $K=k(\nu)\subset\cM$, and has  transcendence degree one over $k$. Then $(K,\tau)\subset(\cM,\tau)$ is a difference field with $k=K^\tau$.

We extend the action of $\tau$ componentwise on matrices.
According to the following result of Praagman, applied to \eqref{eq:sys},
the field $\cM$ is linearly $\tau$-closed:

\begin{thm}[{\cite[Theorem 2]{praagman_fundamental_1986}}]\label{theorem-Praagman}
\label{praagman-theorem}
Consider $M\in\operatorname{GL}_m(\cM)$. There exists $Y\in\operatorname{GL}_m(\cM)$
such that $\tau(Y)=MY$.
\end{thm}

In all the following we write $A=k[t]$, the $k$-algebra of polynomials in an indeterminate $t$ with coefficients in $k$.
We can now introduce the main structure studied in this paper:

\begin{defn}\label{definition:DM}
The {\em Carlitz module} over $\cM$ is the unique homomorphism of $k$-algebras
    $$A\xrightarrow{C}\cM[\tau]$$
such that
    $$C_t:=C(t)=\nu-\tau\in \cM[\tau].$$
\end{defn}

If $a=a_0+a_1t+\cdots+a_dt^d\in A$, the $\tau$-linear operator $C_a=\sum_{i=0}^da_i(\nu-\tau)^i$ can be rewritten in the form
$C_a=(a)_0+(a)_1\tau +\cdots+(a)_d\tau^d\in\cM[\tau]$, for certain elements $(a)_0,\ldots,(a)_d\in k[\nu]$.
Remark that $(a)_d=(-1)^da_d$ and that $(a)_0=\gamma(a)$, where
    $$\gamma:A\rightarrow k[\nu]$$
is the unique $k$-algebra morphism that sends $t$ to $\nu$.
This gives rise to a structure of $A$-module over $\cM$.
If $f$ is a meromorphic function and $a=a_0+a_1t+\cdots+a_dt^d$, we denote by $C_a(f)$ the multiplication of $f$ by $a$ for this $A$-module structure, namely:
     \begin{equation}\label{shape-of-C}
            C_a(f)=(a)_0f+(a)_1\tau(f)+\cdots+(a)_d\tau^d(f)\in\cM.
    \end{equation}
    This expression is also the evaluation of the operator $C_a\in K[\tau]$ at $f$.
We write $C(\cM)$ for the $k$-vector space $\cM$ with the structure of $A$-module described by Definition~\ref{definition:DM}. There are variants of this $A$-module structure that can be associated with other fields $(\cM',\tau)$ and that will be considered later; typically, they are defined over fields $\mathcal{M}'=\mathcal{M}_\Theta$ of meromorphic functions over $\Theta=\CC\setminus S$, with $S$ a locally discrete subset of $\CC$ which is invariant by integer translations, but constructions are essentially the same as
we described here.

The structure of the Carlitz module is deeply intertwined with the derivation in $\nu$. Nothing like this exists in the theory of the ``usual'' Carlitz module so we now spend some time on this.

The derivation
$D=\frac{d}{d\nu}$ of $\cM$ commutes with $\tau$ and we have the following identity of $\CC$-linear endomorphisms of $\mathcal{M}$
    $$D\circ f\tau=D(f)\tau+f\tau\circ D,$$
for any $f\in\cM$.
As far as $A$ is concerned, it is naturally equipped with the derivation $\frac{\partial}{\partial\nu}$, acting on $k$ and trivial on $t$, and the derivation
$\frac{\partial}{\partial t}$, which is trivial on $k$. We also consider the derivation on $A$ defined by $$D(a)=\frac{\partial a}{\partial\nu}+\frac{\partial a}{\partial t}$$ for $a\in A$.
We allowed an abuse of notation which is justified by the following observation. The derivation $D=\frac{d}{d\nu}$ over $\cM$
and the derivation $D=\frac{\partial}{\partial\nu}+\frac{\partial}{\partial t}$ over $A$ are compatible in that, over $A$,
    $$D\circ\gamma=\gamma\circ D$$ (the ``first'' $D$ is over $\mathcal{M}$, the ``second'' over $A$).
We also define an extension of the (trivial) action of $\tau$ from $k$ to $A$ setting $\tau:t\mapsto t+1$ so that we have, similarly,
    $$\tau\circ\gamma=\gamma\circ \tau.$$

\begin{lemma}\label{lemma:basicPropC(M)}
For all $a\in A$ we have:
\begin{enumerate}
  \item $\tau C_a=C_{\tau(a)}\tau$;
  \item $[D,C_a]:=D\circ C_a-C_a\circ D=C_{D(a)}$.
\end{enumerate}
\end{lemma}

\begin{proof}
The first identity is straightforward. Let us prove the second one.
First of all we observe that for any $f\in\cM$ we have:
  \begin{eqnarray*}
    (D\circ(\nu-\tau))(f)&=&D(\nu f-\tau(f))\\
    &=&\nu D(f)+ f-\tau(D(f))\\
    &=&((\nu-\tau)\circ D)(f)+f,
    \end{eqnarray*}
hence
    $$[D,C_t]=D\circ C_t-C_t\circ D=C_{D(t)}=1.$$
If we suppose that for some positive integer $i$ we have
$D\circ C_{t^i}-C_{t^i}\circ D=C_{D(t^i)}=C(it^{i-1})$, then:
    \[
    \begin{array}{rcl}
    D\circ C_{t^{i+1}}
    &=&D\circ C_{t^i} \circ C_t\\
    &=& (C_{t^i}\circ D+C_{D(t^i)})\circ C_t\\
    &=&C_{t^{i+1}}\circ D+C_{t^i}+C_{tD(t^i)}\\
    &=&C_{t^{i+1}}\circ D+C_{D(t^{i+1})}.
    \end{array}
    \]
This proves the second identity for $a=t^i\in A$ by induction for any positive integer $i$.
By linearity, for any $a=\sum_{i=0}^na_it^i\in A$ we have:
    $$
    \begin{array}{rcl}
      D\circ C_a & = & \displaystyle \sum_{i=0}^nD(a_i)(\nu-\tau)^i+ \sum_{i=0}^na_iC_{t^i}\circ D+ \sum_{i=1}^na_iC_{D(t^i)} \\~\\
      & = & C_\frac{\partial a}{\partial\nu}+C_a\circ D + C_{\frac{\partial a}{\partial t}}= C_a\circ D + C_{D(a)}.
    \end{array}
    $$
This ends the proof of the lemma.\end{proof}

Recall that $$\La_a=\ker_k(C_a)=\{f\in \cM: C_a(f)=0\}.$$
By Theorem \ref{theorem-Praagman} applied to the system~\eqref{eq:sys} associated to $C_a$, $\La_a$ is a $k$-vector space of dimension $\deg_t(a)$.
Since for any $m\in\La_a$ and any $b\in A$ we have $C_{ab}(m)=C_a(C_b(m))=C_b(C_a(m))=0$,
one easily shows that the $k$-vector space $\La_a$ is an $A$-module for the structure defined via the morphism $C$.

\begin{defn}
{\em For $a\in A$, the $A$-module $\La_a$ is the {\em $a$-torsion submodule} of $C(\cM)$.}
\end{defn}

\begin{lemma}\label{lemma-deri}
If $a\neq 0$, $\La_a\cap k_\infty=\{0\}$.
\end{lemma}

\begin{proof}
For any $a\in k^\times\subset A$, we have $\La_a=\{0\}$ so we still need to prove the statement under the assumption that $a\notin k$.
Suppose by contradiction that $f\in k_\infty\setminus\{0\}$ is such that there exists $a\in A\setminus k$
with $C_a(f)=0$. This means that there exists $d>0$
such that $\tau^d(f)=f$ and at once the annihilator of $f$ for the Carlitz module structure $C(\cM)$ is non-trivial.
We can suppose that $a$ is a generator of the annihilator of $f$ and that it is monic.
By Lemma~\ref{lemma:basicPropC(M)}, we have
    $$
    C_{\tau^d(a)}(f)=C_{\tau^d(a)}(\tau^d(f))=\tau^d(C_a(f))=0.
    $$
Therefore $\tau^d(a)$ belongs to the annihilator of $f$ as well and $a\mid \tau^d(a)$ in $A$.
Since $a$ and $\tau^d(a)$ have the same degree in $t$ and the same leading coefficient, this implies $a=\tau^d(a)$, hence $a\in k$, which
contradicts the hypothesis.
\end{proof}

We have:

\begin{prop}
\label{prop:descriptionofthetorsion}
Assume that $a\neq 0$. The $A$-module $\La_a$ is isomorphic to $\frac{A}{aA}$.
\end{prop}

\begin{proof}
This is simple but we prefer to give all details because the analogy with the function field theory breaks here somehow, at least if we compare with the treatise of Goss \cite{goss2012basic}.
Without loss of generality we can suppose that $a$ is non-constant and monic.
We know that $\La_a$ is a vector space over $k$ of dimension $\deg_t(a)$ (Theorem \ref{theorem-Praagman}).
The annihilator $\operatorname{Ann}_A(\La_a)=\{b\in A:\La_a\subset\La_b\}$ of $\La_a$ in $A$ is a principal ideal $(Q)$, proper (because it contains $a$), generated by a monic polynomial $Q\in A$ such that $\deg_t(Q)\geq1$, with $Q\mid a$.
In particular, $1\leq \deg_t(Q)\leq \deg_t(a)$. At once $\La_a\subset\La_Q$ which, in terms of dimensions over $k$ and using Theorem \ref{theorem-Praagman} again, implies $\deg_t(a)=\dim_k(\La_a)\leq\dim_k(\La_Q)=\deg_t(Q)$.
Therefore $\deg_t(Q)=\deg_t(a)$, and $(Q)=(a)$. In any case, by the structure theorem of torsion $A$-modules ($A$ is a principal ideal ring) we have an isomorphism of $A$-modules $\La_a\cong\oplus_iA/(a_i)$ for finitely many non-constant monic polynomials $a_i$. Computing annihilators we get the equality
$(a)=\cap_i(a_i)=(\operatorname{lcm}(a_i))$ (monic least common multiple). A last application of Theorem \ref{theorem-Praagman} yields
$\deg_t(a)=\deg_t(\operatorname{lcm}(a_i))\leq\sum_i\deg_t(a_i)=\deg_t(a)$. Therefore $\deg_t(\operatorname{lcm}(a_i))=\sum_i\deg_t(a_i)=\deg_t(a)$ which means that $\operatorname{lcm}(a_i)=\prod_ia_i$ and the polynomials $a_i$ are relatively prime. By the Chinese remainder theorem, $a=\prod_ia_i$ and we have an $A$-module isomorphism $\La_a\cong A/(a)$.
 \end{proof}

We record:

\begin{cor}\label{cor:descriptionofthetorsion}
If $a\in A$ factors as $a=P_1^{r_1}\cdots P_s^{r_s}$ into a product of irreducible polynomials, then
$\La_a\cong\oplus_{i=1}^s\La_{P_i^{r_i}}$ as $A$-modules and $\La_a=\oplus_{i=1}^s\La_{P_i^{r_i}}$ as $A$-submodules of $\mathcal{M}$.
\end{cor}

\begin{proof}
The first property for abstract $A$-modules is obvious. For the second property observe that $\La_{P_i^{r_i}}\subset \La_a$ for all $i$ as submodules of $\mathcal{M}$. These submodules have pairwise trivial intersections because the polynomials $P_i$ are
pairwise coprime. Since $\sum_i\dim_k(\La_{P_i^{r_i}})=\sum_ir_i\deg_t(P_i)=\deg_t(a)=\dim_k(\La_a)$ by Theorem \ref{theorem-Praagman} the $k$-subvector spaces $\La_{P_i^{r_i}}$ span $\La_a$ and $\La_a$ decomposes in the direct sum of its subspaces $\La_{P_i^{r_i}}$.
\end{proof}

Moreover, Proposition~\ref{prop:descriptionofthetorsion} implies the following:

\begin{cor}\label{cor-generators}
For all $a\in A$ the set of generators of $\La_a$ as an $A/(a)$-module is in one-to-one correspondence with the units of the ring $A/(a)$.
\end{cor}

\begin{rmk}\label{rmk:descriptionofthetorsion}
In the language developed by Goss in \cite{goss2012basic}, the elements of $\mathcal{M}$ corresponding to the
the units of the ring $A/(a)$ could be called {\em primitive $a$-division points}.
The corollary above says that, if $m$ is a generator of $\La_a$ as an $A$-module, then the sets of generators of $\La_a$
as an $A$-module is in bijection with
any set of representatives of the classes of invertible
elements of $A$ modulo $(a)$, via the map $b\mapsto C_b(m)$.
A basis of $\La_a$ as a $k$-vector space, on the other hand, is given by $m,C_t(m),\dots,C_{t^{n-1}}(m)$ and therefore
$\La_a=\oplus_{i=0}^{n-1}k\,C_{t^i}(m)$.
The results from Proposition \ref{prop:descriptionofthetorsion} to here are parallel with the results of \cite[\S 3.3]{goss2012basic}. In particular, the $k$-vector spaces $\La_a$ are analogues, in our settings, of the torsion $\FF_q$-vector spaces of ibidem, Definition 3.3.8, and Proposition \ref{prop:descriptionofthetorsion} is the analogue of the arguments described in Goss' Proposition 3.3.8. However, proofs differ for the important fact that
in the whole \S 3.3, Goss uses particular $\infty$-adic uniformizations of Carlitz module structures such as the {\em Carlitz exponential function} of \cite[Theorem 3.2.8]{goss2012basic}, that uniformizes torsion points, unavailable in our setting. This produces a fundamental difference of our work, if compared with the global function field theory.
\end{rmk}

For further reference, we describe two additional properties of $\La_a$. The first one links the torsion of the power of an irreducible polynomial to the action of the derivation $D$.

\begin{prop}\label{prop:PowerOfIrriducibleTorsion}
Let $P\in A$ be monic and irreducible and consider elements $y_1,\dots,y_n\in\cM$ constituting a $k$-basis of $\La_P$.
Then for every integer $r\geq 1$, the meromorphic functions
    $$
    D^j(y_i),\hbox{~for $i=1,\dots,n$ and $j=0,\dots,r-1$,}
    $$
form a $k$-basis of $\Lambda_{P^r}$.
\end{prop}

\begin{proof}
As a first step we prove that for all $r\geq1$, if $y_1,\ldots,y_n$ are as in the hypotheses, then
the elements $D^j(y_i)$ are in $\La_{P^r}$ for $i=1,\ldots,n$ and $j=0,\ldots,r-1$. This is our hypothesis if $r=1$ so we prove
the property by induction. In particular, we can thus suppose that $y_i,\ldots,D^{h-2}(y_i)\in\La_{P^{h-1}}$ for all $i$ for an integer $h>1$. As $\La_{P^{h-1}}\subset\La_{P^h}$ it remains to show that $D^{h-1}(y_i)\in\La_{P^h}$ (for all $i$). We apply
the second identity of Lemma \ref{lemma:basicPropC(M)} which yields:
$$[D,C_{P^h}]=DC_{P^h}-C_{P^h}D=C_{D(P^h)}=hC_{P^{h-1}D(P)}.$$
Now compute, for $i=1,\ldots,n$:
\begin{eqnarray*}
C_{P^h}(D^{h-1}(y_i))&=&C_{P^h}(D(D^{h-2}(y_i)))\\
&=&D(C_{P^h}(D^{h-2}(y_i)))-hC_{D(P)}(C_{P^{h-1}}(D^{h-2}(y_i)))\\
&=&0.
\end{eqnarray*}
Indeed, $D(C_{P^h}(D^{h-2}(y_i)))=0$ by hypothesis and by the fact that $\La_{P^{h-1}}\subset\La_{P^h}$, and
$C_{P^{h-1}}(D^{h-2}(y_i))=0$ by hypothesis. This concludes the first step.

As a second step we prove that for all $r\geq1$, the $rn$ elements $D^j(y_i)$ for $i=1,\ldots,n$ and $j=0,\ldots,r-1$
are linearly independent over $k$. This is our hypothesis for $r=1$. By the first step, these are all elements of $\La_{P^r}$,
and in fact, for all $i,j$, $D^j(y_i)\in\La_{P^{j+1}}$. Let us suppose by contradiction that the property we want to prove is false.
Then, there exists $h>0$ minimal with the property that it also exists a non-trivial $k$-linear combination
$$\ell:=\sum_{i=1}^nc_iD^h(y_i)\in\La_{P^h}$$
(the coefficients $c_i$ are in $k$ and are not all zero). In particular, if $1\leq k<h$
$$\sum_{i=1}^nd_iD^k(y_i)\in\La_{P^{k+1}}\setminus\La_{P^k}$$ for all $d_i\in k$, unless the linear combination is trivial.
Now set
$$\widetilde{\ell}:=\sum_{i=1}^nc_iD^{h-1}(y_i),$$
which belongs to $\La_{P^h}$ by the first step (we use the same coefficients of $\ell$). We have $\widetilde{\ell}\in\La_{P^h}\setminus\La_{P^{h-1}}$. Now, since
$C_{P^h}(\widetilde{\ell})=0$ we have:
\begin{eqnarray*}
0&=&D(C_{P^h}(\widetilde{\ell}))\\
&=&C_{P^h}(D(\widetilde{\ell}))+hC_{P^{h-1}D(P)}(\widetilde{\ell}).
\end{eqnarray*}
Note that by Leibniz rule,
$$D(\widetilde{\ell})=\ell+\sum_{i=1}^nc_iD^{h-1}(y_i).$$ Since $\sum_{i=1}^nc_iD^{h-1}(y_i)\in\La_{P^h}$
we get $C_{P^h}(D(\widetilde{\ell}))$ and $$C_{P^{h-1}D(P)}(\widetilde{\ell})=C_{P^{h-1}}(C_{D(P)}(\widetilde{\ell}))=0.$$
But $P$ is irreducible and $D(P)$ is invertible modulo $P^{h-1}$. This implies that $C_{P^{h-1}}(\widetilde{\ell})=0$ contradicting our hypothesis.

We deduce that the $rn$ elements $D^j(y_i)\in\La_{P^r}$ for $i=1,\ldots,n$ and $j=0,\ldots,r-1$
are linearly independent over $k$. Since $\dim_k(\La_{P^r})=rn$ they are a basis of $\La_{P^r}$.
\end{proof}

\begin{rmk} There is a similar description for $\La_a$ with $a\in A$ general, not just a power of an irreducible polynomial.
\end{rmk}

Le $a$ be an element of $A$. Let $L$ be a field extension of $K$ in $\mathcal{M}$ such that
$\tau$ induces a $k$-automorphism of it.
We write $L\La_a$ for the $L$-span of $\La_a$ in $\mathcal{M}$. It is an $L$-vector space
of dimension $\leq \deg_t(a)$. We also write $L(\La_a)$ for the subfield of $\mathcal{M}$ generated by $L$ and the elements of
$\La_a$.

\begin{lemma}\label{lemma:coprime}
If $a,b\in A$ then
$L(\La_{ab})=L(\La_a)L(\La_b)$.
\end{lemma}

\begin{proof}
By Corollary \ref{cor:descriptionofthetorsion},
$\La_a,\La_b\subset\La_{ab}=\La_a+\La_b$ as $A$-submodules of $\mathcal{M}$. Hence
$$L\La_a,L\La_b\subset L\La_{ab}= L\La_a+L\La_b\subset L(\La_a)L(\La_b).$$
We deduce $$L(\La_a),L(\La_b)\subset L(\La_{ab})\subset L(\La_a)L(\La_b)$$ so we can conclude.
\end{proof}

\begin{lemma}\label{lemma:torsion}
We have that $\tau(L\La_a)=L\La_a$, hence $\tau(L(\La_a))=L(\La_a)$.
\end{lemma}

\begin{proof}
We first observe that $\tau=\nu-C_t$ in $k[\nu][\tau]$. Since $\La_a$ is an $A$-module with the structure induced by the Carlitz module, $\tau(\La_a)\subset L\La_a$. But $\tau(L)=L$ hence
$\tau(L\La_a)=L\tau(\La_a)\subset L\La_a$. In particular $\tau(L\La_a)$ is an $L$-subvector space of $L\La_a$ even though
$\tau$ is not $L$-linear (it is only $k$-linear). Define $W_i:=\tau^i(L\La_a)$ for $i\geq 0$. This gives a flag of
$L$-subvector spaces $W_0\supset W_1\supset\cdots$. Since they all have finite dimension, there exists $j\geq 0$ such that
$W_j=W_{j+1}$. In particular, the $k$-vector space $W_j/W_{j+1}$ is trivial. But $\tau$ is $k$-linear and injective on $L\La_a$ and induces $k$-isomorphisms $W_i/W_{i+1}\rightarrow W_{i+1}/W_{i+2}$ for all $i$ including $i=0$. This means that $W_0=W_1$ and $\tau(L\La_a)
=L\La_a$. We conclude observing that $\tau$ extends to a $k$-automorphism of $L(\La_a)$.\end{proof}

\begin{cor}\label{cor:torsion}
We have, for $a\in A$, that $L\La_{a}=L\La_{\tau(a)}$ and $L(\La_a)=L(\La_{\tau(a)})$.
\end{cor}

\begin{proof}
By the first part of Lemma \ref{lemma:basicPropC(M)} we have that for $f\in\mathcal{M}$,
$C_a(f)=0$ if and only if $C_{\tau(a)}(\tau(f))=0$. Hence $$\tau(\La_a)=\La_{\tau(a)}.$$
By Lemma \ref{lemma:torsion} we deduce the equalities $L\La_a=L\tau(\La_a)=L\La_{\tau(a)}$
and $L(\La_a)=L(\La_{\tau(a)})$.
\end{proof}

We consider the equivalence relation $\sim$ on monic polynomials of $A$ defined by $a\sim b$ if and only if there exists $r\in\ZZ$ such that $a=\tau^r(b)$.
Let $a\in A$ be a monic non-constant polynomial, with factorization
$$a=\prod_{i\in I}P_i^{r_i},$$ where
the polynomials $P_i\in A$ are distinct, irreducible and monic, and the exponents $r_i$ are $>0$.
The set $\{P_i:i\in I\}$ is partitioned by $\sim$. Associated with this partition we have a unique decomposition in disjoint union $I=\sqcup_{j\in J}I_j$so that $\{P_i:i\in I\}=\sqcup_j\{P_i:i\in I_j\}$ and for all $j\in J$, the set $\{P_i:i\in I_j\}$ is an equivalence class for $\sim$. For all $j$, write $r_j=\max\{r_i:i\in I_j\}$. There exists $Q_j\in\{P_i:i\in I_j\}$ unique with the following property. For all $i\in I_j$ there exists $h_i\in\NN$ such that $\tau^{h_i}(P_i)=Q_j$. Note that if $P_i=Q_j$ then $h_i=0$.

We write
    \begin{equation}\label{eq:defHatA}
    \hat a=\prod_{j\in J}Q_j^{r_j}.
    \end{equation}
The polynomial $\hat{a}\in A$, monic, is uniquely determined. Note that $\hat{a}\mid a$. We also extend to the constant monic polynomial writing $\hat{1}=1$.

\begin{lemma}\label{lemma:TranscendenceDegree}
For all $a\in A$ and for any $L$ subfield of $\mathcal{M}$ containing $K$
such that $\tau(L)=L$
we have $L\La_a=L\La_{\hat{a}}$ and $L(\La_a)=L(\La_{\hat a})$.
\end{lemma}

\begin{proof}

In the notation introduced above, for each $i\in I$ there exist a unique
$j\in J$ and $h_i\in\NN$ such that $\tau^{h_i}(P_i)=Q_j$.
It follows from Corollary~\ref{cor:torsion} that
$L\La_{P_i^{r_i}}= L\La_{\tau^{h_i}(P_i^{r_i})}\subset L\La_{Q_j^{r_j}}$ (recall that if $a\mid b$ then $\La_a\subset\La_b$).

By Corollary \ref{cor:descriptionofthetorsion}, $\La_a=\oplus_i\La_{P_i^{r_i}}$ as $A$-submodules of $\mathcal{M}$.
Similarly, $\Lambda_{\hat a}=\oplus_{j\in J}\Lambda_{P_j^{r_j}}$. Hence we have the identities of $L$-subvector spaces of $\mathcal{M}$:
 $$
    L\Lambda_a=\sum_{i\in I}L\Lambda_{P_i^{r_i}}=\sum_{i\in I}L\Lambda_{\tau^{h_i}(P_i)^{r_i}}
    \subset\sum_{j\in J}L\Lambda_{P_j^{r_j}}=L\Lambda_{\hat a}\subset L\Lambda_{a},
 $$
 the first inclusion follows from the fact that $r_i\leq r_j$ for all $i\in I_j$ and the last inclusion is derived from the fact that $\hat{a}\mid a$. Hence $L\La_a=L\La_{\hat{a}}$ and
 $L(\Lambda_a)= L(\Lambda_{\hat a})$.
\end{proof}

\section{Torsion extensions and their difference Galois groups}
\label{sec:GaloisTheory}

Choose $a\in A$ monic. Recall that $\hat a$ is defined in \eqref{eq:defHatA}. The following is one of the crucial results of this paper:

\begin{thm}\label{thm:TrdegPV}
The field $K(\La_a)$ is a purely transcendental extension
of $K$ of transcendence degree $n$, where $n=\deg_t(\hat a)$.
More precisely, if $y_1,\dots, y_n\in\cM$ form a basis of $\La_{\hat a}$ over $k$,
then they also are a basis of transcendence of  $K(\La_a)$ over $K$.
\end{thm}

\begin{rmk}
As a consequence of Theorem \ref{thm:TrdegPV}, given $P\in A$ irreducible, for all $r\geq0$ the natural map $K\otimes_k\La_{P^r}\rightarrow K\La_{P^r}$
is an isomorphism of $K$-vector spaces and for all $a\in A$, if $\hat{a}=\prod_jP_j^{r_j}$ with the polynomials $P_j$ irreducible and distinct,
$K\La_a=\oplus_jK\La_{P_j^{r_j}}$ as $K$-subvector spaces of $\mathcal{M}$.
\end{rmk}

To prove Theorem \ref{thm:TrdegPV} it suffices to consider $a$ with $a=\hat{a}$, indeed
in Lemma \ref{lemma:TranscendenceDegree} we proved that for any $a\in A$ monic the
fields $K(\La_a)$ and $K(\La_{\hat a})$ are equal.
By Theorem~\ref{praagman-theorem}, we can choose a $k$-basis $y_1,\dots,y_n$ of $\La_{\hat{a}}$ in $\cM$.
We consider the $K$-algebra
    \begin{equation}\label{eq:PV meromorfo}
      R_{\hat{a}}:=K\l[\tau^j(y_i),~i=1,\dots,n,~j=0,\dots, n-1\r]\!\l[\delta^{-1}\r],
    \end{equation}
where $\delta$ is the determinant of $\Cas_\tau(y_1,\ldots,y_n)$.
Using Lemma~\ref{lemma:torsion}, one sees that $R_{\hat{a}}=K[y_1,\ldots,y_n,\delta^{-1}]$ independently on the choice of the $k$-basis.

\begin{defn}
{\em We define the \emph{$\tau$-difference Galois group} of the extension $K(\La_{\hat{a}})$ over $K$ to be
the algebraic group scheme (see for instance \cite[Definition 2.4.3]{divizioDifferenceGaloisTheory2021}):
    $$
    \begin{array}{rcccl}
    G_{\hat{a}}:&\{\hbox{\rm commutative $k$-algebras}\} & \xrightarrow{\Phi_{\hat{a}}} &\{\hbox{\rm Groups}\},&\\ \\
    &S & \longmapsto & \hbox{Aut}^{\tau}_{S\otimes_k K}(S\otimes_k R_{\hat a})&
    \end{array}
    $$
where $\hbox{Aut}^{\tau}_{S\otimes_k K}(S\otimes_k R_{\hat a})$ is the group
of $S\otimes_k K$-algebra automorphisms of $S\otimes_k R_{\hat a}$ commuting with $\tau$, where $\tau$ is trivial on $S\otimes k$.}
\end{defn}

Notice that the $G_{\hat{a}}$ is a linear algebraic group
(see \cite[Thm.~1.13]{putGaloisTheoryDifference1997}, see also
\cite[\S2.7]{ovchinnikovGaloisTheoryLinear2014}, where
one can reduce to our situation supposing that $\sigma$
is the identity).
Since $k$ is not algebraically closed,
the group $G_a(k)$ can be  ``smaller than expected'' and for this reason we
use the language of group schemes.
The following statement on the structure of the Galois groups appearing in this theory is a consequence of Theorem~\ref{thm:TrdegPV}.
As explained already, it is a first step towards a difference class field theory using an approach
in the spirit of \cite{hayes1974explicit}:

\begin{thm}\label{thm:Artinsymbol}
The $\tau$-difference Galois group $G_{\hat{a}}$ is canonically isomorphic to the functor $\GG_m(A/(\hat{a}))$. In particular, $G_{\hat{a}}$ is commutative.
\end{thm}

First of all we show that Theorem~\ref{thm:TrdegPV} implies Theorem~\ref{thm:Artinsymbol}. We will prove Theorem~\ref{thm:TrdegPV}
in the next section.

\begin{proof}[Proof of Theorem~\ref{thm:Artinsymbol} (assuming Theorem~\ref{thm:TrdegPV}).]
For any $k$-algebra $S$ we construct an isomorphism of groups:
    $$
    \begin{array}{rcccc}
    \Phi_{\hat{a}}(S):&G_{\hat{a}}(S) & \longrightarrow &\ds\l( \frac{S\otimes_kA}{(1\otimes \hat a)}\r)^\times
    &\ds\cong \l(S\otimes_k\frac{A}{(\hat a)} \r)^\times
    \end{array},
    $$
compatible with morphisms of $k$-algebras.
By Corollary~\ref{cor-generators}, there exists
$m\in\La_{\hat a}$ generating $\La_{\hat a}$ as an $A$-module, hence such that
    \begin{equation}\label{a-basis}
    m,C_t(m),\dots,C_{t^{n-1}}(m)
    \end{equation}
is a basis of $\La_{\hat a}$ over $k$.
By Theorem~\ref{thm:TrdegPV} these elements are also algebraically independent over $K$.
Lemma~\ref{lemma:tauwronskian} implies that $1\otimes m,1\otimes C_t(m),\dots,1\otimes C_{t^{n-1}}(m)$ is a basis
of $S\otimes_k\La_{\hat a}$ as an $S$-module. We extend the $k$-algebra map $C:A \to k[\nu][\tau]$ to a $S\otimes1$-algebra map $S\otimes_k A\rightarrow (S\otimes_k k[\nu])[\tau]$ by $S\otimes1$-linearity. To simplify the notation, we
denote it by $C$ again.
\par
If $\sg\in G_{\hat{a}}(S)$, then $\sg(1\otimes m)\in S\otimes_k\La_{\hat a}$ and there exist
$b_0,\dots, b_{n-1}\in S$ such that
$b_\sigma=\sum_{i=1}^{n-1}b_it^i\in S\otimes_k A$ satisfies:
    $$
    \sg(1\otimes m)=\sum_{i=0}^{n-1}b_iC_{t^i}(1\otimes m)=C_{b_\sigma}(1\otimes m),
    $$
Now we write:
    $$
    \Phi_{\hat{a}}(S)(\sg):=C_{b_\sg}\in (S\otimes_kk[\nu])[\tau].
    $$
This determines an injective group homomorphism
    $$
    G_{\hat{a}}(S)\xrightarrow{\Phi_{\hat{a}}(S)}\ds\l( \frac{S\otimes_kA}{(1\otimes\hat a)}\r)^\times.
    $$
To conclude, it suffices to show that $\Phi_{\hat{a}}(S)$ is also surjective.
Consider an element $b\in (S\otimes_kA/(1\otimes\hat a))^\times$
corresponding to the automorphism $C_b$ of \hbox{$S\otimes_k\La_{\hat{a}}$}.
Since $m,C_t(m),\dots,C_{t^{n-1}}(m)$ are algebraically independent over $K$ by Theorem \ref{thm:TrdegPV},
they are linearly independent and the assignements
    $$\sg_b(C_{t^i}(1\otimes m))=C_b\big(C_{t^i}(1\otimes m)\big)$$
for all $i$ extend uniquely to a $K$-automorphism of $K(\La_{\hat a})$. This automorphism commutes with $\tau$.
Indeed we have
    $$
    \begin{array}{rcl}
    \sg_b(\tau(m))&=&\sg_b(\nu m-C_t(m))=\nu C_b(m)-C_bC_t(m)\\~\\
    &=&\nu C_b(m)-C_t C_b(m)=(\nu-C_t)(C_b(m))=\tau(\sg_b(m)),
    \end{array}
    $$
and similar relations hold when one applies $\sigma_b$ to the other elements of the basis given by $m,C_t(m),\dots,C_{t^{n-1}}(m)$. This induces an automorphism of
$M:=K[m,C_t(m),\dots,C_{t^{n-1}}(m)]$ commuting with $\tau$ (defining an endomorphism of this $k$-algebra).

It remains to show that $\sigma_b(R_{\hat a})=R_{\hat a}$. By our choice of basis and \eqref{eq:sys},
$$\tau(\delta)=(-1)^n\hat{a}(\nu)\delta.$$
Hence
    $$
    \tau\big(\sg_b(\delta)\big)=\sg_b\big(\tau(\delta)\big)=(-1)^n\hat{a}(\nu)\sg_b\big(\delta\big).
    $$
Therefore,
$\tau(\frac{\sg_b(\delta)}{\delta})=\frac{\sg_b(\delta)}{\delta}$ and
$$\frac{\sg_b(\delta)}{\delta}\in k^\times.$$ Since $R_{\hat a}=M[\delta^{-1}]$ we conclude that $\sigma_b$ defines an automorphism of $R_{\hat a}$
which commutes with $\tau$. The functor $S\otimes_k-$ is flat and $\sg$ naturally defines an
automorphism of $S\otimes_kR_{\hat a}$, commuting with $\tau$.
This proves that $\Phi_a(S)$ is a group isomorphism.
\par
We have constructed $\Phi_{\hat{a}}$ making the choice of the $k$-basis (\ref{a-basis}). We are going to show that its definition is in fact independent on the choice of the basis and hence, that
$\Phi_{\hat{a}}$ is canonical.

Consider any $S$-basis of $S\otimes_k\La_{\hat a}$ with elements $y_1,\ldots,y_n$. We can find, for all $i=1,\ldots,n$,
$c_i=\sum_jc_{i,j} t^j\in S\otimes_kA$ such that $y_i=C_{c_i}(m)$. We can also suppose that $\deg_t(c_i)\leq n-1$ for all $i$.
If $\sg\in G_{\hat{a}}(S)$, by definition it commutes with $\tau$ and acts as the identity over $S\otimes_k K$.
Let $b\in S\otimes_kA$ be such that $\sg(m)=C_b(m)$. Then for all $i=1,\dots,n$ we have:
    \begin{multline*}
    \sg(y_i)=\sg\l(\sum_{j=0}^{d-1}c_{i,j}C_{t^j}(m)\r)=
    \sum_{j=0}^{d-1}c_{i,j}\sg\l(C_{t^j}(m)\r)=\\
    =\sum_{j=0}^{d-1}c_{i,j}C_{t^j}\l(\sg(m)\r)=\sum_{j=0}^{d-1}c_{i,j} C_b\big(C_{t^j}(m)\big)=
    C_b(y_i).
    \end{multline*}
This concludes the proof of the Theorem.
\end{proof}

\section{Proof of Theorem~\ref{thm:TrdegPV}}

In this section we keep considering $a\in A$ monic and we write $L_a$ for the field $K(\La_a)$ to simplify the notation. We recall that it coincides with $L_{\hat a}$.
Let $p\in A$ be the product of the irreducible factors of $\hat{a}$, so that if $\hat{a}=P_1^{r_1}\cdots P_d^{r_d}$ with the exponents $r_i$ that are $>0$, then
$p=P_1\cdots P_d$. We recall that for all $m\in\ZZ$, $P_i$ and $\tau^m(P_j)$ are relatively prime.
We set
        $$
        L_{p^\infty}:=K(\La_{P_i^{r}}; i=1,\dots, d,\,r\in\Z_{\geq 1})=\bigcup_{r\geq 1}K(\La_{p^{r}})
        $$
        (the second equality follows from Lemma \ref{lemma:coprime}).
This is the subfield of $\cM$ generated by $K$ and $\La_{P_i^r}$ for all $i$ and $r\geq 1$.
It follows from Proposition~\ref{prop:PowerOfIrriducibleTorsion} that $L_{p^\infty}$ is stable by both $\tau$ and $D$.
By construction we have $L_{\hat{a}}\subset L_{p^\infty}$.
Our purpose is to prove the differential independence of a set of generators of $\La_p$,
which implies Theorem~\ref{thm:TrdegPV}:

\begin{prop}\label{prop:pinfty}
In the notation above, for all $i=1,\ldots,d$ we consider a basis $y_{i,1},\dots,y_{i,n_i}$ of $\La_{P_i}$ over $k$,
where $n_i=\deg_t(P_i)$.
The set
    $$
    \l\{D^h(y_{i,j}); i=1,\dots,d,\, j=1,\dots,n_i,\,h\geq 0\r\}
    $$
is algebraically independent over $K$ and is a basis of transcendence of $L_{p^\infty}$ over $K$.
\end{prop}
Proposition~\ref{prop:pinfty} implies Theorem~\ref{thm:TrdegPV} because the set
        $$
        \l\{D^h(y_{i,j}); i=1,\dots,d,\, j=1,\dots,n_i,\,h=0,\dots,r_i-1\r\}
        $$
generates $L_{\hat{a}}$ over $K$.

\begin{proof}[Proof of Proposition~\ref{prop:pinfty}.]
By \cite[Definition~1.7 and Theorem~2.2]{woodDifferentiallyClosedFields1998} there exists a differentially closed field
extension $\tilde{k}$ of $k$ with respect to $D=\frac{d}{d\nu}$. We refer to this work for the precise definition of
differentially closed field, but it roughly means that every algebraic differential equation with coefficients in $k$
that has a non-trivial solution in a given $k$-algebra equipped with an extension of $D$ also has a non-trivial solution in $\tilde{k}$.
In particular, $\tilde{k}$ is algebraically closed. Note that $\nu$ being solution of $D(x)=1$, there are embeddings of $K$ in $\tilde{k}$.
\par
We need to construct a $\tau$-field extension $\tilde{K}/K$, equipped with a derivation $D$ commuting with $\tau$ and such that $\tilde{k}=\tilde{K}^\tau$,
and a $\tilde{K}$-algebra $F$ also equipped with a commuting extension of $\tau$ and $D$, such that
$\tilde{k}=F^\tau$, in which the solutions $D^h(y_{i,j})$ make sense, so that we can apply the following proposition to
a convenient choice of $\tau$-equations that we will construct later:

\begin{prop}[{\cite[Prop.~3.1]{hardouinDifferentialGaloisTheory2008}. See also \cite[Theorem~4.2]{hardouincompositio}}]\label{prop:HS}
Let $a_1,\dots,a_n\in\tilde{K}$ and $\omega_1,\dots,\omega_n\in F$ be satisfying
$\tau(\omega_i)-\omega_i=a_i$, for all $i = 1,\dots, n$.
Then $\omega_1,\dots,\omega_n$ are differentially dependent over $\tilde{K}$ if and only if there exists a nonzero
homogeneous linear differential polynomial $L(Y_1,\dots, Y_n)$ with coefficients in $\tilde{k}$ and
an element $g\in\tilde{K}$ such that
    \[
    L(a_1,\dots, a_n) = \tau(g)-g.
    \]
\end{prop}

We construct $\tilde{K}$ as follows. We consider the tensor product $\tilde{k}\otimes_k K$, on which
we define the action of $\tau$ as $\operatorname{Id}_{\tilde{k}}\otimes\tau$, where $\operatorname{Id}_{\tilde{k}}$ is the identity on $\tilde{k}$. Note that $\tilde{k}\otimes_k K$ is a domain (in general, if $L,L'$ are field extensions of $F$ and $L'$ is purely transcendental over $F$, then $L\otimes_FL'$ is a domain). In particular the field of fraction of $\tilde{k}\otimes_k K$, well defined, is isomorphic to $\tilde{k}(x)$ with $x$ an indeterminate that corresponds to $1\otimes\nu$.

To simplify our notation in what follows, we denote by $\nu$ the element $1\otimes \nu$
and we set $$\tilde{K}:=\tilde{k}(\nu),$$
a field of transcendence degree one over $\tilde{k}$. This should be written $(\tilde{k}\otimes1)(1\otimes\nu)$ but in the following we never use the copies of $K$ existing in $\tilde{k}\otimes1$ so we can relax notations in the way indicated above.
\par
We consider the multiplicative subset
$S=\tilde{k}[\nu]\smallsetminus\{0\}$, so that $\tilde{K}$ is isomorphic to $S^{-1}(\tilde{k}\otimes_k K)$ (in fact, it is also isomorphic to $S^{-1}(\tilde{k}\otimes_kk[\nu])$).
The embedding of $k$-algebras $K\hookrightarrow L_{p^\infty}$ induces an embedding
$$\tilde{k}\otimes_k K\hookrightarrow\tilde{k}\otimes_k L_{p^\infty},$$ extending $1\otimes K\rightarrow1\otimes L_{p^\infty}$
(but $\tilde{k}\otimes_k L_{p^\infty}$ is not necessarily a domain).
In particular the image of the multiplicative subset $S$ via this embedding in $\tilde{k}\otimes_k L_{p^\infty}$ is a multiplicative subset
and it makes sense to consider the $\tilde{K}$-algebra $$F:=S^{-1}(\tilde{k}\otimes_k L_{p^\infty}).$$
We now draw a preliminary picture of the structure of $F$, quite subtle, in view of the results we are targeting. We are going to make use of some technical results contained in \cite{ovchinnikovGaloisTheoryLinear2014}.

\begin{lemma}\label{lemma-subtle}
The $\tilde{K}$-algebra $F$ is equipped with commuting extensions of $\tau$ and $D$ such that $F^\tau=\tilde{k}$.
\end{lemma}

\begin{proof}
The operator $\tau$ is defined on $\tilde{k}\otimes_k L_{p^\infty}$ as $\operatorname{Id}_{\tilde{k}}\otimes\tau$.
Since $D=\frac{d}{d\nu}$ acts on both $\tilde{k}$ and  $L_{p^\infty}$, a derivation that we still call $D$ is defined on $\tilde{k}\otimes_k L_{p^\infty}$
using the Leibniz rule, namely $D(f\otimes m)=D(f)\otimes m+f\otimes D(m)$,
for any $f\in\tilde{k}$ and $m\in L_{p^\infty}$.
These operators extend in a unique way to $F$ and one checks that they commute.
\par
We need to prove that $F^\tau=\tilde{k}$. Fix an integer $n\geq1$.
We recall that $L_{p^n}=K(\La_{p^n})$, so that $L_{p^\infty}=\cup_{n\geq 1}L_{p^n}$.
We denote by $\delta_n$ the determinant of a Casorati matrix associated to a chosen basis of $\La_{p^n}$.
The ring
    $$R_{p^n}:=K[\La_{p^n},\delta_n^{-1}]\subset\mathcal{M},$$
already considered in \eqref{eq:PV meromorfo},
is a finitely generated $K$-algebra contained in $L_{p^n}$. It is also stable by the action of $\tau$ and hence a $\tau$-ring.
Moreover $k\subset R_{p^n}^\tau\subset L_{p^n}^\tau=k$ and $R_{p^n}$ is obviously a domain.
By \cite[Prop.~2.14]{ovchinnikovGaloisTheoryLinear2014} we have that $R_{p^n}$ is $\tau$-simple.
By \cite[Lemma~2.9]{ovchinnikovGaloisTheoryLinear2014} we also verify that $\tilde{k}\otimes_kR_{p^n}$ is $\tau$-simple.
Finally applying \cite[Lemma~2.7]{ovchinnikovGaloisTheoryLinear2014} we obtain that the field of constants of the
total quotient ring of $\tilde{k}\otimes_kR_{p^n}$
satisfies:
  \begin{equation}\label{intermediate-id}
    \operatorname{Quot}(\tilde{k}\otimes_kR_{p^n})^\tau=(\tilde{k}\otimes_kR_{p^n})^\tau.
  \end{equation}
We claim that:
    \begin{equation}\label{eq:claim}
    \hbox{$S^{-1}(\tilde{k}\otimes_kR_{p^n})$ is a $\tau$-simple finitely generated $\tilde{K}$-algebra.}
    \end{equation}
Assume the claim verified for a moment. Then \cite[Lemma~2.13]{ovchinnikovGaloisTheoryLinear2014} ensures that
$\big(S^{-1}(\tilde{k}\otimes_kR_{p^n})\big)^\tau$ is an algebraic extension of $\tilde{k}$, which is in particular algebraically closed, being differentially closed.
Therefore $\big(S^{-1}(\tilde{k}\otimes_kR_{p^n})\big)^\tau=\tilde{k}$ for all $n\geq 1$. This implies that
$\big(S^{-1}(\tilde{k}\otimes_kL_{p^n})\big)^\tau=\tilde{k}$.
Indeed observe:
$$\tilde{k}\otimes_kR_{p^n}\subset S^{-1}(\tilde{k}\otimes_kR_{p^n})\subset S^{-1}(\tilde{k}\otimes_kL_{p^n})\subset\operatorname{Quot}(\tilde{k}\otimes_kR_{p^n}).$$
Computing $\tau$-constants, we obtain
$$(\tilde{k}\otimes_kR_{p^n})^{\tau}\subset \tilde{k}=\big(S^{-1}(\tilde{k}\otimes_kR_{p^n})\big)^\tau\subset \big(S^{-1}(\tilde{k}\otimes_kL_{p^n})\big)^{\tau}\subset\operatorname{Quot}(\tilde{k}\otimes_kR_{p^n})^\tau.$$
By (\ref{intermediate-id}) we get
$$\operatorname{Quot}(\tilde{k}\otimes_kR_{p^n})^\tau=(\tilde{k}\otimes_kR_{p^n})^{\tau}=\tilde{k}.$$
Now consider $\frac{f}{g}\in F^\tau=(S^{-1}(\tilde{k}\otimes_kL_{p^\infty}))^\tau$, with $g\in S$ and $f\in \tilde{k}\otimes_kL_{p^n}$ for a certain choice of $n\geq 1$.
Then $\frac{f}{g}\in (S^{-1}(\tilde{k}\otimes_kL_{p^n}))^\tau=\tilde{k}$. Therefore $F^\tau=\tilde{k}$ and would be the conclusion of the proof.

We are left with proving the claim \eqref{eq:claim}. All we need to prove is that $\big(S^{-1}(\tilde{k}\otimes_kR_{p^n})\big)$ is $\tau$-simple and this can be deduced from \cite[Proposition 3.11]{atiyahIntroductionCommutativeAlgebra1969}, from which we also borrow the terminology.
Consider the localization map
$$\tilde{k}\otimes_kR_{p^n}\rightarrow S^{-1}(\tilde{k}\otimes_kR_{p^n}),$$
injective by the fact that $S\subset \tilde{K}$ is constituted by nonzero divisors. Every ideal $\tilde{I}$ of $S^{-1}(\tilde{k}\otimes_kR_{p^n})$ is extension of an ideal $I$ of $\tilde{k}\otimes_kR_{p^n}$, $\tilde{I}=S^{-1}I$, as explained in ibidem. But every ideal is contained in the contraction of its extension, therefore $I\subset\tilde{I}\cap(\tilde{k}\otimes_kR_{p^n})$. Suppose that $\tilde{I}$ is a $\tau$-ideal. Then the contraction ideal $\tilde{I}\cap(\tilde{k}\otimes_kR_{p^n})$ is a $\tau$-invariant ideal of $\tilde{k}\otimes_kR_{p^n}$ which is proven to be $\tau$-simple (see the above). If $\tilde{I}\cap(\tilde{k}\otimes_kR_{p^n})=(0)$ then $I=(0)$ and
$\tilde{I}=(0)$. Otherwise $\tilde{I}$ contains $\tilde{k}\otimes_kR_{p^n}$ which means that it equals $S^{-1}(\tilde{k}\otimes_kR_{p^n})$.
\end{proof}

We go back to the polynomial $p=P_1\cdots P_d\in A$.
We look at $C_p$ as an operator in $\tilde K[\tau]$ where $\tau$ acts $(\tilde{k}\otimes1)$-linearly. Recall that to simplify the writing, we denote by $\zeta$ any element of the form $\zeta\otimes1$ with $\zeta\in\tilde{k}$, and by $\nu$ the element $1\otimes\nu$.
For  $i=1,\dots,d$ and $j=1,\dots,n_i$, if $\zeta_{i,j}$ are the roots of $P_i$ in $\tilde k$, then $C_p$ can be factorized as
    $$
    C_p=\prod_{i,j}(C_t-\zeta_{i,j}).
    $$
The key-point here that the $\zeta_{i,j}$ commute with the operator $C_t$, since $C_t$ is $(\tilde{k}\otimes1)$-linear
like $\tau$. Therefore
the factors $C_t-\zeta_{i,j}$ commute pairwise and seeing them as endomorphisms of $\tilde{k}\otimes_k\La_p$, they are simultaneously diagonalizable (this can be viewed easily from the fact that $C_p$ is zero on $\tilde{k}\otimes_k\La_p$).  It follows that for any $i=1,\dots,d$ and $j=1,\dots,n_i$ there exists an eigenvector
$\omega_{i,j}\in \tilde{k}\otimes_k\La_p$ such that $\omega_{i,j}$ satisfies $C_t(\omega_{i,j})=\zeta_{i,j}\omega_{i,j}$, or equivalently $\tau(\omega_{i,j})=(\nu-\zeta_{i,j})\omega_{i,j}$.
The above commutativity property implies that the $\omega_{i,j}$ and the $y_{i,j}$ both define bases of $\tilde{k}\otimes_k\La_p$ over $\tilde{k}$.
With tensor products in clear, we have the relations $\tau(\omega_{i,j})=(1\otimes \nu-\zeta_{i,j}\otimes 1)\omega_{i,j}$.
Notice that the $\omega_{i,j}$ are naturally constructed in $F$ but need not to belong to $1\otimes_k L_{p^\infty}$, unlike the elements $y_{i,j}$.

\par
The following remark is fundamental in the lines below: the elements $\omega_{i,j}$ are units of $\tilde{k}\otimes L_{p}$.
Indeed, $\tilde{k}\otimes_k L_{p}$ is a simple $\tau$-ring (use for example \cite[Lemma 2.8]{ovchinnikovGaloisTheoryLinear2014} to check this property) and $(\omega_{i,j})$ is a $\tau$-ideal,
since $\tau(\omega_{i,j})=(\nu-\zeta_{i,j})\omega_{i,j}\in (\omega_{i,j})$,
for any $i,j$.
Since by construction $\omega_{i,j}\neq 0$ and because of the $\tau$-simplicity, the $\tau$-ideal $(\omega_{i,j})$ is the whole
$\tilde{k}\otimes_k L_{p}$, which proves that $\omega_{i,j}$ is invertible in $\tilde{k}\otimes_k L_{p}$ and hence in $F$.
\par
Since the $\omega_{i,j}$ are invertible, we can consider the following elements:
    $$w_{i,j}=\frac{D(\omega_{i,j})}{\omega_{i,j}}\in F^\times.$$
By the fact that $D$ commutes with $\tau$, taking logarithmic derivatives of both sides the difference equations
    $$\tau(\omega_{i,j})=(\nu-\zeta_{i,j})\omega_{i,j},$$
we verify that the elements $w_{i,j}$ satisfy the system:
    $$
    \l\{\tau(w_{i,j})=w_{i,j}+\frac{1-D(\zeta_{i,j})}{\nu-\zeta_{i,j}}\r.,~\forall i=1,\dots,d,\hbox{~and~} j=1,\dots,n_i,
    $$
and the elements $\omega_{i,j}\in F$ are differentially independent over $\tilde{K}$ if and only if
the elements $w_{i,j}\in F$ are differentially independent over $\tilde{K}$.

We suppose by contradiction that the $w_{i,j}$ are differentially dependent over $\tilde{K}$.
All the above comments including Lemma \ref{lemma-subtle} allow to apply directly Proposition \ref{prop:HS}
(that is, \cite[Prop.~3.1]{hardouinDifferentialGaloisTheory2008}) which says that there exist
$\lambda_{i,j,r}\in\tilde{k}$ not all zero and
$g\in \tilde{K}$ such that

    \begin{equation}\label{identity}
    \sum_{\begin{smallmatrix}i=1,\dots,d\\ j=0,\dots,n_i\\ r\geq0\end{smallmatrix}}\lambda_{i,j,r}D^r\l(\frac{1-D(\zeta_{i,j})}{\nu-\zeta_{i,j}}\r)=\tau(g)-g,
    \end{equation}
where there are finitely many non-zero coefficients $\lambda_{i,j,r}\in \tilde{k}$.
We stare at this identity as one in $\tilde{K}$ which is, as we already noticed, a purely transcendental extension of $\tilde{k}$ of transcendence degree one,
where as above we continue to identify $\nu$ with $1\otimes \nu$ and $\zeta_{i,j}$ with $\zeta_{i,j}\otimes1$.

Since for all $r\geq 0$ the denominator of $D^r\l(\frac{1-D(\zeta_{i,j})}{\nu-\zeta_{i,j}}\r)$ as a reduced fraction in $\tilde{K}$ has the form $(\nu-\zeta_{i,j})^ {r+1}$,
the subset $\mathcal{E}$ of $\tilde{k}$ of poles in $\nu$ of the left hand side is a non-empty subset of the set of all the $\zeta_{i,j}$.
In particular, if $\zeta,\zeta'\in\mathcal{E}$ are distinct, then their difference $\zeta-\zeta'$ does not belong to $\ZZ$ (recall that
the $\zeta_{i,j}$ are the zeroes of $p$ which is constructed in such a way that if $P,P'$ are distinct irreducible dividing it,
then for all $m\in\ZZ$, $P'$ and $\tau^m(P)$ are relatively prime).

We now move to analyzing the right-hand side of \eqref{identity}. Let $\mathcal{F}$ be the set of poles in $\tilde{k}$
of the expression in the right-hand side, which is an element of $\tilde{K}$.
Obviously $\mathcal{E}=\mathcal{F}$ hence $\mathcal{F}$ is non-empty.
Now, by the fact that
$\tilde{k}$ is algebraically closed, we can expand $g$ in a sum of elementary fractions:
    $$
    g=Q+\sum_{i,j}\frac{a_{i,j}}{(\nu-\eta_i)^j},
    $$
where $Q\in \tilde{k}[\nu]$, $a_{i,j},\eta_j\in\tilde{k}$ for all $i,j$, and $j>0$. Then
    $$
    \tau(g)=Q+\sum_{i,j}\frac{a_{i,j}}{(\nu+1-\eta_i)^j}.
    $$
Since the set $\mathcal{G}:=\{\eta_j\}$ is finite, for any of its elements $\eta$ the set of $m\in\ZZ$ such that $\eta+m\in\mathcal{G}$ is finite so that if $\eta\in\mathcal{F}$, then there exists $m\neq 0$ such that also $\eta+m$
is in $\mathcal{F}$. However, $\mathcal{F}=\mathcal{E}$. This means that there exist $\zeta,\zeta'\in\mathcal{E}$ distinct with
$\zeta-\zeta'\in\ZZ$ which is impossible. Hence the $w_{i,j}$ are differentially independent over $\tilde{K}$.
\par
Since the $\omega_{i,j}$ are invertible in $F$, the subset of $F$
    $$
    Z=\l\{D^r(z_{i,j}); i=1,\dots,d,\, j=1,\dots,n_i,\,r\geq 0\r\}.
    $$
is algebraically dependent over $\tilde{K}$ if and only if
the $w_{i,j}$ are differentially dependent. We conclude that $Z$ is algebraically independent over $\tilde{K}$.
\par
The $\omega_{i,j}$ and the $y_{i,j}$ are both bases of the same $\tilde{k}$-vector space and
we deduce that the set $\l\{D^r(y_{i,j}); i=1,\dots,d,\, j=1,\dots,n_i,\,r\geq 0\r\}$
is algebraically independent over $\tilde{K}$ and a fortiori over $K$.
This ends the proof of the proposition. This also concludes the proof of Theorem \ref{thm:TrdegPV}.
\end{proof}

A curious very simple but nontrivial application is the following.

\begin{cor}
Let $P,Q\in A$ be monic irreducible. Then $K\La_P=K\La_Q$ if and only if there exists $k\in\ZZ$ such that
$\tau^k(P)=Q$.
\end{cor}

\begin{proof}
By Corollary \ref{cor:torsion} if there exists $k\in\ZZ$ such that
$\tau^k(P)=Q$, then $K\La_P=K\La_Q$. Now assume by contradiction that $K\La_P=K\La_Q$ but for all $k$,
$P$ and $\tau^k(Q)$ are coprime. Then we can apply Theorem \ref{thm:TrdegPV} with $p=PQ$ and the field
$K(\La_p)$ has transcendence degree $\deg_t(p)=\deg_t(P)+\deg_t(Q)$. By Lemma \ref{lemma:torsion} and
Corollary \ref{cor:torsion} however, $K(\La_p)=K(\La_P)$ that, again by Theorem \ref{thm:TrdegPV} has transcendence degree $
\deg_t(P)$, which is impossible.
\end{proof}

\section{Ax-Lindemann-Weierstrass theorem related to \texorpdfstring{$\Gamma$}{gamma}}
\label{ALW-gamma}

We fix $\zeta_1,\dots,\zeta_n$ in an algebraic closure of $k$ and we choose a suitable non-empty connected open domain $\Omega\subset \C$ on which
 $\zeta_1,\dots,\zeta_n$ can be identified with analytic functions $\zeta_1(\nu),\dots,\zeta_n(\nu)$, in such a way that also the functions
    \begin{equation}\label{eq:Gammas}
      \Gamma(\nu-\zeta_1(\nu)),\dots,\Gamma(\nu-\zeta_n(\nu))
    \end{equation}
are analytic.
We show:

\begin{thm}\label{thm:Holder-ALW}
Let us suppose that $\zeta_1,\dots,\zeta_n$ are pairwise distinct modulo $\Z$.
Then the functions $ \Gamma(\nu-\zeta_1(\nu)),\dots,\Gamma(\nu-\zeta_n(\nu))$ are differentially independent over $K$.
\end{thm}

\begin{rmk}
If $n=1$ and $\zeta_1\in\Z$, our result implies a stronger version of Hölder's Theorem for the gamma function which can be also deduced directly from \cite[Cor.~3.4]{hardouinDifferentialGaloisTheory2008} applied to
the functional equation $\Gamma(\nu+1)=\nu\Gamma(\nu)$. The statement \cite[Cor.~3.4]{hardouinDifferentialGaloisTheory2008} is a corollary of
\cite[Prop.~3.1]{hardouinDifferentialGaloisTheory2008}, a stronger result that we used crucially in the proof of Theorem~\ref{thm:TrdegPV}. It seems unlikely to deduce our Theorem easily and directly from \cite[Cor.~3.4]{hardouinDifferentialGaloisTheory2008}. We use Theorem \ref{thm:TrdegPV} and the theory of the Carlitz module over meromorphic functions.
\end{rmk}

\begin{proof}[Proof of Theorem~\ref{thm:Holder-ALW}.]
Without loss of generality we can suppose that the set $\{\zeta_1,\ldots,\zeta_n\}$
is closed under the action of the automorphisms of the absolute Galois group of $k$. Indeed, given $\zeta$ algebraic
over $k$, no one of its translated by a non-zero integer can be conjugated to it. This can be resumed in the sentence: ``there are no non-trivial Artin-Schreier extensions in characteristic zero''.

The proof of our theorem results from intermediate lemmas 
presented below.
Let
    $$
    p=a_0+a_1t+\cdots+a_{n-1}t^{n-1}+t^n\in A
    $$
be a polynomial of minimal degree such that $p(\zeta_{i})=0$, for all $i$.
By assumption $p=\hat p=P_1\cdots P_d$, where the polynomials $P_i\in A$ are irreducible.
We set $n_i=\deg_t(P_i)$ and we rename the $\zeta_i$ as $\zeta_{i,j}$ so that for any $i=1,\dots,d$ the element $\zeta_{i,j}$ is a root of $P_i$ for $j=1,\dots,n_i$. From now on we identify algebraic elements over $k$ with analytic functions over
open subsets of $\CC$ therefore we sometimes omit the variable $\nu$ in writing such functions.
We set:
    $$
    y_{i,r}=\sum_{j=1}^{n_i}\zeta_{i,j}^r\Gamma(\nu-\zeta_{i,j}),
    \hbox{for $i=1,\dots,d$ and $r=0,\dots, n_i-1$.}
    $$
Consider $\epsilon>0$.
We denote by $S_\epsilon$ the disjoint union of disks of radius $\epsilon$ centered at the poles
of the coefficient $a_i$ of $p$ (we can choose $\epsilon$ small so this happens) and by $S$ the set of poles of the $a_i$. By the fact that the functions $a_i$ are $1$-periodic, $S_\epsilon$ and $S$ are $\tau$-invariant.
The $y_{i,r}$, unlike the $\zeta_{i,j}$, are uniform on $$\Theta_\epsilon:=\C\smallsetminus S_\epsilon,$$ meaning that they have trivial monodromy. More precisely, setting also $\Theta:=\CC\setminus S$, we have:

\begin{lemma}\label{lemma-tau-S}
For all $\epsilon>0$ small, the functions $y_{i,r}$ admit a meromorphic continuation to $\Theta_\epsilon$,
for $i=1,\dots,d$ and $r=0,\dots, n_i-1$, and their set of poles is discrete in the adherence of $\Theta_\epsilon$. In particular,
the functions $y_{i,r}$ extend to meromorphic functions over $\Theta$.
\end{lemma}

\begin{proof}
Since $\Gamma$ is a meromorphic function over $\C\smallsetminus\Z_{<0}$ and it has simple poles at
each point of $\Z_{<0}$, the functions $y_{i,r}$ have poles in the set:
    \[
    Z:=\{x\in\Theta_\epsilon: \exists (i,j) \hbox{~such that~} \zeta_{i,j}(x)-x\in\Z_{>0}\}
    \]
    (depending on $\epsilon$).
We also need to study the behavior of $y_{i,r}$ around the set $R$ of ramification points of the $\zeta_{i,j}$ which are not poles of the $a_i$.
The $a_i$ being $1$-periodic functions, both $Z$ and $R$ are $\tau$-invariant.
The set $R$ is discrete since it is the set of the $x\in\Theta_\epsilon$ such that $p$ has multiple roots when its coefficient are specialized at $x$.
\par
Suppose that we have already proved that $Z$ is discrete.
For any $x_0\in\Theta_\epsilon\setminus(Z\cup R)$ there exists an open neighborhood $\Omega\subset\Theta_\epsilon\setminus(Z\cup R)$ of
$x_0$ over which each root $\zeta_{i,j}$ extends to $\Omega$ analytically. Let $\gamma\subset\Theta_\epsilon\setminus(Z\cup R)$ be a closed path from $x_0$ to $x_0$, possibly
``turning around'' some points of $Z\cup R$. By construction, both the $\zeta_{i,j} $ and the $\Gamma(\nu-\zeta_{i,j})$
can be analytically continued along $\gamma$.
The uniqueness of analytic continuation implies that the analytic continuation of each $\zeta_{i,j}$ along $\gamma$
is still a root of $P$ and distinct roots cannot coincide after analytic continuation:
In other words, the action of the
monodromy can only permute the roots, so that the analytic continuation of the function  $y_{i,r}$ on any closed path, whose constituting terms are permuted,
has same initial and final values. Since the $y_{i,r}$ have trivial monodromy at the points of $R$ where the $\zeta_i$ have finite values,
the $y_{i,r}$ are analytic at the points of $R$.
\par
We still need to prove that $Z$ is discrete in $\Theta_\epsilon$. Let $x\in\Z$ and let us suppose by contradiction that there exists an open disk $D$,
centered at $x$ and contained in $\Theta_\epsilon$, such that $D\cap Z$ is infinite.
Since there are only finitely many $\zeta_{i,j}$, there exists a sequence of points $x_k$, $k\geq 0$, and $\iota=(i,j)$,
such that  $\zeta_\iota(x_k)-x_k=N_k\in\Z_{>0}$. Reordering the sequence and extracting a subsequence, we can suppose that the $N_k$ form an ascending sequence
of integer and that the sequence $x_k$ actually admits a limit $x_\infty$ in the adherence of $D$. The limit of $\zeta_\iota(x_k)-x_k$ is therefore $\infty$, which means that
$\zeta_\iota$ has a pole at $x_\infty$. We conclude that $x_\infty$ is the center of one of the disks composing $S_\epsilon$, which means that all but finitely many
points of the sequence $x_k$ do not belong to $\Theta_\epsilon$. We have obtained a contradiction.
The above discussion being valid for all $\epsilon>0$ small, taking the limit $\epsilon\rightarrow0$ we deduce that the functions $y_{i,r}$ are meromorphic over $\Theta$.
\end{proof}

In particular, the set of poles of the functions $y_{i,r}$ is locally discrete in $\Theta$. We point out that the functions $y_{i,r}$ may have essential singularities in $S$.

The automorphism $\tau$ extends to an automorphism of the field of meromorphic functions $\mathcal{M}_{\Theta}$ over $\Theta$.
The field $\mathcal{M}_{\Theta}$ contains $\cM$ and hence $k$, and its field of constants $k_\Theta$ is the field of $1$-periodic functions that are meromorphic over $\Theta$ and possibly have
essential singularities in $S$, the extension $k_\Theta/k$ may not be algebraic. The identity function $\nu$ is transcendental over $k_\Theta$. It makes sense to say that an element of  $\cM_{\Theta}$ is a solution of $C_p(y)=0$.

\begin{lemma}\label{lemma-y-solutions}
For any $i=1,\dots,d$ and $r=0,\dots, n_i-1$, the functions $y_{i,r}$ are solutions of $C_p(y)=0$.
\end{lemma}

\begin{proof}
Let us consider an open disk $\Xi$ contained in $\Theta$ where the $\zeta_{i,j}$ and the $\Gamma(\nu-\zeta_{i,j})$ are all analytic and consider $x_0\in\Xi$.
Modifying $\Xi$ we can suppose that $\Xi\cap\tau^k(\Xi)=\emptyset$ for all $k\neq0$.
Choose a continuous path $\gamma:[0,1]\to\Theta$ avoiding the poles of the functions $\Gamma(\nu-\zeta_{i,j})$ and
such that $\gamma(0)=x_0$ and $\gamma(1)=x_0+1\in\Xi+1=\tau(\Xi)$.  We consider the analytic continuations of the $\zeta_{i,j}$
over an open ``tube'' $\Xi_\gamma$ from $\Xi$
to $\tau(\Xi)$ containing $\gamma$ and avoiding singularities. We keep denoting by $\zeta_{i,j}$ such analytic continuations.
Since the coefficients of $p$ are $1$-periodic there exist, for $i=1,\dots,d$, permutations $\sigma_i$ of $\{1,\dots,n_i\}$ such that we have:
    \[
    \zeta_{i,j}(\nu+1)=\zeta_{i,\sigma_i(j)}(\nu),\hbox{~for all $\nu\in\Xi$.}
    \]
For any $i=1,\dots,d$ and $r=0,\dots, n_i-1$ we have:
    $$
    \begin{array}{rcl}
          \tau(y_{i,r})
    &=&\displaystyle\sum_{j=1}^{n_i}\zeta_{i,j}^r(\nu+1)\Gamma(\nu+1-\zeta_{i,j}(\nu+1))\\~ \\
    &=&\displaystyle\sum_{j=1}^{n_i}\zeta_{i,\sigma_i(j)}^r(\nu)(\nu-\zeta_{i,\sigma_i(j)}(\nu))\Gamma(\nu-\zeta_{i,\sigma_i(j)}(\nu))\\~\\
    &=&\displaystyle\sum_{j=1}^{n_i}\zeta_{i,j}^r(\nu)(\nu-\zeta_{i,j}(\nu))\Gamma(\nu-\zeta_{i,j}(\nu)).
    \end{array}
    $$
Therefore
\begin{equation}\label{Ct-y}C_t(y_{i,r})=\sum_{j=1}^{n_i}\,\zeta_{i,j}^{r+1}(\nu)\,\,\Gamma(\nu-\zeta_{i,j}(\nu))
\end{equation} and:
    $$
    C_p(y_{i,r})=\sum_{j=1}^{n_i}\,\zeta_{i,j}^r(\nu)\,p(\zeta_{i,j}(\nu))\,\Gamma(\nu-\zeta_{i,j}(\nu))=0.
    $$
Since $\gamma$
is arbitrary, $C_p(y_{i,r})=0$ (over $\Theta$).
\end{proof}

We consider functions $\zeta_1,\ldots,\zeta_n$ that are algebraic over $k$ with the property that if $i\neq j$, then $\zeta_i$ and $\zeta_j+s$
are not conjugate over $k$ for all $s\in\ZZ$. Let $\Theta\subset\CC$ be as in Lemma \ref{lemma-tau-S}. We can find a non-empty open subset $\Omega$ of $\Theta$ such that $\tau^N(\Omega)=\Omega$ for some $N>0$, over which the functions $\zeta_i$ and $\Gamma(\nu-\zeta_i)$ are  analytic for all $i$. Additionally we can choose $N$ so that $\tau^N(\zeta_i)=\zeta_i$ for all $i$. Let $\ell$ be the field of
meromorphic functions $f$ over $\Omega$ such that there exists $n>0$, depending on $f$, with $\tau^n(f)=f$ and $N\mid n$.
Clearly, $\ell$ contains the $1$-periodic meromorphic functions over $\Omega$ and all the $\zeta_i$. Moreover $\nu$ is transcendental over the field $\ell$.

\begin{lemma}\label{lemma1}
Under the assumptions of Theorem~\ref{thm:Holder-ALW},
the functions $\Gamma(\nu-\zeta_1),\ldots,\Gamma(\nu-\zeta_m)$ are linearly independent over $\ell(\nu)$.
\end{lemma}

\begin{proof}
Suppose that there is a non-trivial linear dependence relation between functions over $\Omega$:
\begin{equation}\label{initial-id}
\sum_{i=1}^m\lambda_i\Gamma(\nu-\zeta_i)=0\end{equation}
with coefficients $\lambda_i\in\ell(\nu)$. Without loss of generality we can suppose that $\lambda_i\neq0$ for all $i$, that $\lambda_m=1$, and that the relation is of minimal length with the non-triviality property.
Recall that for $n\geq0$, $$\tau^n(\Gamma)=(\nu)^{(n)}\Gamma$$
where $(x)^{(n)}=x(x+1)\cdots(x+n-1)$ is the ascending Pochammer polynomial of degree $n$.
If $\zeta$ is algebraic over $k$ and analytic over $\Omega$, a simple computation tells us that
$$\tau^n(\Gamma(\nu-\zeta))=(\nu-\tau^{n}(\zeta))^{(n)}\Gamma(\nu-\tau^n(\zeta)),$$
where $\tau(\zeta),\tau^2(\zeta),\ldots$ are conjugate of $\zeta$ over $k$ and analytic over $\Omega$. These functions belong to $\ell$.
Our hypothesis implies that if $i\neq j$ the polynomials
$(\nu-\tau^{a}(\zeta_i))^{(a)},(\nu-\tau^{b}(\zeta_j))^{(b)}$ are relatively prime for all $a,b$. Indeed otherwise they would have a common linear factor,
implying that $\zeta_i$ and $\zeta_j+s$ are conjugate for some $s\in\ZZ$ and $i\neq j$ against our assumptions.

Apply $\tau^n$ on both sides of (\ref{initial-id}). We get:
$$\sum_{i=1}^m\tau^n(\lambda_i)(\nu-\tau^n(\zeta_i))^{(n)}\Gamma(\nu-\tau^n(\zeta_i))=0.$$
Supposing further that $N\mid n$ we get, over $\Omega$:
$$\sum_{i=1}^m\tau^n(\lambda_i)(\nu-\zeta_i)^{(n)}\Gamma(\nu-\zeta_i)=0.$$

Dividing by $(\nu-\zeta_m)^{(n)}$ and subtracting we obtain
an $\ell(\nu)$-linear dependence relation connecting the functions $\Gamma(\nu-\zeta_i)$ of length $<m$ so that it is trivial. This means that for all $n$,
$$\lambda_i=\tau^n(\lambda_i)\frac{(\nu-\zeta_i)^{(n)}}{(\nu-\zeta_m)^{(n)}},\quad i=1,\ldots,m-1,$$
but this is impossible. If $s:\ell(\nu)^\times\rightarrow\NN$ is the function (size) that associates to a non-zero rational function $\frac{a}{b}$
with $a,b$ coprime in $\ell[\nu]$ the sum of the degrees of $a$ and $b$, we see that there exists an integer $n_0$
such that $s\left(\frac{\lambda_i}{\tau^n(\lambda_i)}\right)\leq n_0$ for all $n$ with $N\mid n$, while $s\left(\frac{(\nu-\zeta_i)^{(n)}}{(\nu-\zeta_m)^{(n)}}\right)$ tends to infinity
as $n$ increases, due to the hypothesis on $\zeta_1,\ldots,\zeta_m$. Hence $\lambda_1=\ldots=\lambda_m=0$
and the functions $\Gamma(\nu-\zeta_i)$ are linearly independent over $\ell(\nu)$.
\end{proof}

Thanks to the previous Lemma we can now prove that the functions $y_{i,r}$ determine a $k_\Theta$-basis of solutions of the equation $C_p(y)=0$ in $\mathcal{M}_{\Theta}$:

\begin{lemma}\label{linear-independence}
Under the assumptions of Theorem~\ref{thm:Holder-ALW}, the functions $y_{i,r}$, for $i=1,\dots,d$ and $r=0,\dots, n_i-1$, are $k_\Theta$-linearly independent.
\end{lemma}

\begin{proof} By Lemma \ref{lemma1} the functions $\Gamma(\nu-\zeta_{i,j})$ are $k_\Theta$-linearly independent.
For any $i=1,\dots,d$, denote by $y_i$ the column vector with $n_i$ entries $y_{i,0},\ldots,y_{i,n_{i}-1}$ in $\mathcal{M}_{\Theta}$. We have, by definition:
    \[
    y_i=
    \begin{pmatrix}
      1 & \cdots & 1 \\
      \zeta_{i,1} & \cdots & \zeta_{i,n_i} \\
      \vdots & \cdots & \vdots\\
      \zeta_{i,1}^{n_i-1} & \cdots & \zeta_{i,n_i}^{n_i-1}\\
    \end{pmatrix}
    \begin{pmatrix}\Gamma(\nu-\zeta_{i,1})\\\vdots\\\Gamma(\nu-\zeta_{i,n_i})\end{pmatrix}.
    \]
    If we now write $y$ for the column vector obtained by concatenating the vectors $y_1,\ldots,y_d$ and by $\gamma$
    the column vector with entries the functions $\Gamma(\nu-\zeta_{i,j})$ ordered in compatibility with the construction of
    $y$, we have
    \begin{equation}\label{yMg}y=M\gamma\end{equation}
    where $M$ is a block diagonal matrix and the blocks on the diagonals are Vandermonde matrices.
The $\zeta_{i,j}$ being pairwise distinct, this matrix is invertible, hence its columns are linearly independent over $k$ and the entries $y_{i,r}$ of $y$ are linearly independent over $k$, too.
So the Casorati matrix associated to the functions $y_{i,r}$ ($i$ fixed) is invertible, which means, by Lemma \ref{lemma:tauwronskian}, that they are not only linearly independent over $k$, but also over the meromorphic
$1$-periodic functions over $\Theta$, since the $y_{i,r}$ are themselves meromorphic over $\Theta$.
\end{proof}

\subsubsection*{End of proof of Theorem \ref{thm:Holder-ALW}.}
After the above discussion of preliminary properties of the basis $y=(y_{i,r})$ of solutions of $C_p(y)=0$, we can
conclude the proof of Theorem \ref{thm:Holder-ALW}.

By Theorem~\ref{praagman-theorem} (Praagman's Theorem) and using that the Casorati matrix of
a basis of solutions of the homogeneous linear difference equation $C_p(y)=0$ is an invertible solution of a difference system of the form
(\ref{eq:sys}), for all $i=1,\dots,d$ there exist
an invertible square matrix $H_i$ of order $n_i$, whose entries are $1$-periodic meromorphic functions over $\Theta$, such
that
    $$
    m_i:=(m_{i,r}):=H_iy_i
    $$
with the entries $m_{i,r}$ of $m_i$ that are meromorphic functions over $\CC$ determining, with $i$ running from $1$ to $d$, a $k$-basis of solutions of $C_p(y)=0$ in $\mathcal{M}$. We also write $m$ for the column vector obtained by concatenation of the vectors $m_1,\ldots,m_d$ in the appropriate order.
It follows from Theorem~\ref{thm:TrdegPV} that the coefficients of $m$, that is the functions $m_{i,r}$, and all their derivatives, are algebraically independent over $K$; we want
to obtain a similar conclusion for the $y_{i,r}$.
\par
If we prove that the $y_{i,r}$ are differentially independent over $K$ then the
same holds for the functions $\Gamma(\nu-\zeta_{i,j})$, $i=1,\dots,d$ and $j=1,\dots,n_i$, thanks to (\ref{yMg}) and the already noticed invertibility of the matrix $M$.

Suppose by contradiction that the functions $y_{i,r}$ are differentially dependent over $K$.
We shall now deduce that the functions $m_{i,r}$ are
differentially dependent over $K$ as well.
\par
Denote by $k_H$ the extension of $k$ in $\mathcal{M}_\Theta$ generated by all the entries of all the matrices $H_i$ as well as
all their derivatives (so this field is invariant under $D=\frac{d}{d\nu}$). Now observe:

\begin{lemma}\label{lemma-independent-M-k}
If a set of meromorphic $1$-periodic functions over $\Theta$ is linearly independent over $k$, then it is
linearly independent over $\cM$.
\end{lemma}

\begin{proof}
Let $f_1,\dots,f_n$ be $1$-periodic functions meromorphic over $\Theta$ which are linearly dependent over $\cM$.
Let $\la_i\in\cM$ be such that $\sum_i\la_i f_i=0$. Suppose that the functions $f_i$ are linearly independent over $k$,
which implies that at least one of the functions $\la_i$ belongs to $\cM\smallsetminus k$.
Then we can assume $n$ to be minimal for such a property.
Hence, we can assume that the $\la_i$ are all non zero and rescaling, we can also suppose that at least one of them is equal to $1$.
Applying $\tau$ and subtracting the two relations so obtained, one concludes that $\sum_i(\la_i-\tau(\la_i))f_i=0$.
By the minimality of $n$, the functions $\la_i$ all belong to $k$ contradicting the hypothesis that the $f_i$ are linearly independent over $k$.
\end{proof}

We come back to the proof of our Theorem and we conclude it.
Our hypothesis by contradiction on the functions $y_{i,r}$ implies that the functions $m_{i,r}\in\cM$, entries of the vector $m$, are
differentially dependent over $k_H(\nu)$.
Let $\cP$ be a non-trivial algebraic differential operator with coefficients in $k_H(\nu)$ such that
\begin{equation}\label{vanishing-Pm}\cP(m)=0.\end{equation}
We can identify, up to multiplication by a common denominator in $k_H[\nu]$,
$\mathcal{P}$ with a polynomial with coefficients in $k_H[\nu]$ in the independent indeterminates $x_{i,r,j}$
with $1\leq i\leq d$, $0\leq r\leq n_i-1$ and $j\geq0$. The condition (\ref{vanishing-Pm}) amounts to the property that,
replacing $x_{i,r,j}$ with $D^j(m_{i,r})$ for all $i,r,j$, in $\mathcal{P}(x_{i,r,j})$, produces zero identically.
The coefficients of $\mathcal{P}$ can be identified with finitely many polynomials in $k_H[\nu]$. The coefficients of these polynomials, in $k_H$, span a finitely dimensional $k$-subspace of $k_H$ of which we choose a basis $(b_i)_{i\in I}$, with
$I$ finite set of indices; the elements of this basis are $1$-periodic functions of $\mathcal{M}_\Theta$. Thanks to the fact that $k_H$ and $K$ are linearly disjoint in $\mathcal{M}_{\Theta}$ (remember that
$\nu$ is transcendental over $k_H$) we can decompose, uniquely,
$$\mathcal{P}=\sum_ib_i\mathcal{P}_i$$
with $\mathcal{P}_i\in k[\nu][x_{i,r,j}]$, and we have at least one term that does not vanish identically because $\mathcal{P}$ is non-trivial by assumption. Evaluating at $m$ yields:
$$0=\mathcal{P}(m)=\sum_ib_i\mathcal{P}_i(m).$$
But $\mathcal{P}_i(m)\in\mathcal{M}$. By Lemma \ref{lemma-independent-M-k}, $\mathcal{P}_i(m)=0$ for all $i$.
This means that the entries of $m$ are differentially dependent over $K$ contradicting our hypotheses.

This ends the proof of Theorem \ref{thm:Holder-ALW}.
\end{proof}


\end{document}